\documentclass[10pt]{article} 

\usepackage[utf8]{inputenc} 

\usepackage{geometry} 
\usepackage{graphicx} 

\usepackage{booktabs} 
\usepackage{array} 
\usepackage{paralist} 
\usepackage{verbatim} 
\usepackage{subfig} 
\usepackage[all]{xy}
\usepackage{amsmath} 
\usepackage{amsthm}
\usepackage{amsfonts}
\usepackage{bm}
\usepackage{longtable}

\usepackage{fancyhdr} 
\usepackage{sectsty}
\allsectionsfont{\sffamily\mdseries\upshape} 

\usepackage[nottoc,notlof,notlot]{tocbibind} 
\usepackage[titles,subfigure]{tocloft} 

\title{Galois covers of type $(p,\cdots,p)$, vanishing cycles formula, and  the existence of torsor structures.}
\author{Mohamed Sa\"\i di \& Nicholas Williams}
\date{} 

\newcommand{\defeq}{:=}
\newcommand{\Spec}{\mathrm{Spec}}
\newcommand{\chr}{\mathrm{char}}

\newtheorem*{corB}{Corollary 1.5}

\newtheorem*{corE}{Corollary 2.2.3}

\newtheorem*{proA}{Proposition 2.2.2}
\newtheorem*{proB}{Proposition 3.7}

\newtheorem*{defA}{Definition 1.3}

\newtheorem*{defB}{Definition 2.1.1}

\newtheorem*{lemA}{Lemma 1.4}

\newtheorem*{thmA}{Theorem 1.1}
\newtheorem*{thmB}{Theorem 1.2}
\newtheorem*{thmE}{Theorem 1.6}
\newtheorem*{thmF}{Theorem 2.1.2}
\newtheorem*{thmC}{Theorem 2.2.1}
\newtheorem*{thmD}{Theorem 3.2}
\newtheorem*{thmG}{Theorem 3.5}
\newtheorem*{thmH}{Theorem 3.4}

\newtheorem*{queA}{Question 3.1}
\newtheorem*{queB}{Question 3.3}

\begin{document}
\maketitle

\begin{abstract} In this article we prove a local Riemman-Hurwitz formula which compares the dimensions of the spaces of vanishing cycles in a finite Galois cover
of type $(p,p,\cdots,p)$ between formal germs of $p$-adic curves and which generalises the formula proven in [Sa\"\i di1] in the case of Galois covers of degree $p$. 
We also investigate the problem of the existence of a torsor structure for a finite Galois cover 
of type $(p,p,\cdots,p)$ between $p$-adic schemes.
\end{abstract}

\section*{\S 0. Introduction}
Let $K$ be a {\it complete discrete valuation ring of mixed characteristic}, $R$ its valuation ring, and
$k\defeq R/\pi R$ the residue field of characteristic $p>0$ which we assume to be {\it algebraically closed}. 
We suppose that $K$ contains a primitive $p$-th root of $1$.
In [Sa\"\i di1] the first author proved a local Riemman-Hurwitz formula which compares the dimensions of the spaces of vanishing cycles in a finite Galois cover of degree $p$ 
between formal germs of $R$-curves. This formula is quite explicit and involves the (usual) {\it \lq\lq generic\rq\rq different}, which measures the ramification at the level of generic fibres,
and  a certain {\it \lq\lq special\rq\rq different} which involves certain {\it \lq\lq conductors\rq\rq} attached to the induced covers between the formal boundaries of the formal germs 
(cf. loc. cit. Theorem 3.4).

In this paper we generalise this formula to the setting of Galois covers of type $(p,p,\cdots,p)$, i.e., with Galois group $\Bbb Z/p\Bbb Z\times \cdots\times \Bbb Z/p\Bbb Z$.
In principle one can apply the formula in the Galois degree $p$ case obtained in [Sa\"\i di1] iteratively to derive such a formula. However, the difficulty here lies in computing the {\it conductors involved in the special different at the various degree $p$ intermediate covers}; the possibility of having generically purely inseparable extensions at the level of special fibres doesn't allow the use of the standard ramification theory as in [Serre] in order to compute these conductors. In this paper we are able to compute in {\S 1} these conductors at the various degree $p$ intermediate levels via direct, rather tedious, computations (cf. Theorem 1.1). Although our main result computing these conductors is stated only in the case of Galois covers of type $(p,p)$ (cf. loc. cit.), it is quite straightforward to deduce from this result the relevant value conductors as well as the corresponding Riemman-Hurwitz formula in the case of general Galois covers of type $(p,p,\cdots,p)$ (cf. Example 1.7 for an illustration). In $\S2$ we derive an explicit local Riemman-Hurwitz formula which compares the dimensions of the spaces of vanishing cycles in a finite Galois covers
of type $(p,p)$ between formal germs of $R$-curves, which can be easily iterated to deduce a similar local Riemman-Hurwitz formula in the general case of finite Galois covers
of type $(p,p,\cdots,p)$.  

In $\S3$ we investigate the problem of the existence of a torsor structure for a finite Galois cover 
of type $(p,p,\cdots,p)$ between $R$-(formal) schemes where we allow $R$ to be of equal characteristic $p>0$. 
Let $X$ be a {\it normal flat and geometrically connected} $R$-(formal) scheme with an {\it integral special fibre} $X_k\defeq X\times _Rk$,
$\{f_i:Y_i\to X\}_{i=1}^n$ torsors under finite and flat $R$-group scheme $G_i$ which are generically pairwise disjoint, $1\le i\le n$, and $f:Y\to X$ the morphism of normalisation of $X$ in (the fibre product over $X_K\defeq X\times _RK$) $\prod _{i=1}^nY_{i,K}$, where $Y_{i,K}\defeq Y_i\times_RK$. Assume the special fibre $Y_k\defeq Y\times _Rk$ is {\it reduced}.
Our main result Theorem 3.4 gives necessary and sufficient conditions for $f$ to have the structure of a torsor under a finite and flat $R$-group scheme (necessarily isomorphic to 
$G_1\times_R \cdots\times_R G_n$). In the case where $X$ is a {\it relative curve} these conditions are equivalent to the condition that at least $n-1$ of the group schemes $G_i$ are \'etale (cf. Theorem 3.5).
This latter fact is false in relative dimension $>1$ (cf. 3.8). 

\section*{Notations} 

In this paper $p\ge 2$ is a {\it prime integer},
$K$ is (unless we specify otherwise) a {\it complete discrete valuation ring}, $\mathrm char (K)=0$, $R$ its valuation ring, $\pi$ a uniformising parameter, 
$v_K$ will denote the valuation of $K$ which is normalised by $v_K(\pi)=1$, and
$k\defeq R/\pi R$ the residue field of characteristic $p>0$ which we assume to be {\it algebraically closed}. We suppose $R$ contains a primitive $p$-th root of $1$.

For an $R$-scheme $X$ we will denote by
$X_K\defeq X\times _RK$ (resp. $X_k\defeq X\times _Rk$) the {\it generic} (resp. {\it special}) fibre of $X$.
If $X=\mathrm{Spf} A$ is a formal affine $R$-scheme we will denote $X_K\defeq \mathrm {Spec} (A\otimes_RK)$
and $X_k\defeq \mathrm {Spec}(A/\pi)$ the special fibre of $X$.
 
A {\it formal (resp. algebraic) $R$-curve} is an $R$-formal scheme of finite type (resp. scheme of finite type)
flat, separated, and whose special fibre is equidimensional of dimension $1$.

We will refer to a (generically separable) cover $Y\to X$ between normal connected (formal $R$-)schemes which is Galois with Galois group 
$\Bbb Z/p\Bbb Z\times \Bbb Z/p\Bbb Z$ as a Galois cover of type $(p,p)$.

Let $X$ be a {\it proper, normal, (formal) $R$-curve} with $X_k$ {\it geometrically reduced}.
For $x\in X$ a {\it closed point}
let $F_x=\mathrm {Spf} (\hat {O}_{X,x})$ be the {\it formal completion} of $X$ at $x$, which we will refer to as the {\it formal germ} of $X$ at $x$.
Thus, $\hat {O}_{X,x}$ is the completion of the local ring of (the algebraisation) of $X$ at $x$.
Let $\{P_i\}_{i=1}^n$ be the {\it minimal prime ideals of $\hat {O}_{X,x}$ which contain $\pi$}; they correspond 
to the {\it branches} $\{\eta_i\}_{i=1}^n$ of the completion of $X_k$ at $x$ (i.e., closed points of the normalisation of $X_k$ above $x$), 
and $X_i=X_{x,i}\defeq \mathrm {Spf} (\hat {O}_{x,P_i})$
the formal completion of the localisation of $F_x$ at $P_i$. The local ring $\hat {O}_{x,P_i}$ is a {\it complete 
discrete valuation ring} with uniformiser $\pi$.
We refer to $\{X_{i}\}_{i=1}^n$ as the set of {\it boundaries of the formal germ} $F_x$. 
We have a canonical morphism $X_{i}\to F_x$ of formal schemes, $1\le i\le n$.

With the same notations as above, let $x\in X$ be a {\it closed} point and $\widetilde X_k$ the {\it normalisation} of $X_k$. There is a one-to-one correspondence
between the set of points of $\widetilde X_k$ above $x$ and the set of boundaries of the formal germ at the point $x$. Let $x_i$ be the point
of $\widetilde X_k$ above $x$ which corresponds to the boundary $X_{i}$, $1\le i\le n$.
Then the completion of $\widetilde X_k$ at $x_i$ is isomorphic to the spectrum of a ring of formal power series $k[[t_i]]$ over $k$,
where $t_i$ is a {\it local parameter} at $x_i$.  
The complete local ring $\hat {O}_{x,P_i}$ is a discrete valuation ring with uniformiser $\pi$, and residue field isomorphic
to $k((t_i))$. Fix an isomorphism $ k((t_i))\simeq \hat {O}_{x,P_i}/\pi$.
Let $T_i\in \hat {O}_{x,P_i}$ be an element which lifts (the image in $\hat {O}_{x,P_i}/\pi$ under the above isomorphism of)
$t_i$; we shall refer to such an element $T_i$ as a {\it parameter} of 
$\hat {O}_{x,P_i}$, or of the boundary $X_{i}$.
Then there exists an isomorphism $R[[T_i]]\{T_i^{-1}\}\simeq    \hat {O}_{x,P_i}$, where  
$$R[[T]]\{T^{-1}\} \defeq \Big\{\sum _{i=-\infty}^{\infty}a_iT^i,\ \underset 
{i\to -\infty} {\mathrm lim} \vert a_i \vert=0 \Big\}$$ 
and $\vert \ \vert$ is a normalised absolute value of $R$ (cf. [Bourbaki], $\S2$, 5).

Given a power series $g\in k((z))$ where $z$ is an indeterminate we write 
$$g(z)=\sum _{i\in I\subset \Bbb Z} a_iz^i+\text{higher order terms},$$
meaning all remaining monomial terms in $z$ are of the form $cz^t$
where $c\in k$ and $t>i$ for at least one $0\neq i\in I$.  Also given a power series $H(Z)\in R[[T]]\{T^{-1}\}$
we write 
$$H(Z)=(F(Z))^p+\sum _{i\in I\subset \Bbb Z} c_iZ^i+\text{higher order terms},$$
meaning all remaining monomial terms in $Z$ are of the form $dZ^t$
where either $v_K(d)>v_K(c_i)$ for all $i\in I$ or there exists at least one $i\in I$
such that $v_K(d)=v_K(c_i)$ and $t>i$.

\section*{Background}
In this section we collect/improve some background material form [Sa\"\i di1] that will be used in this paper.
Let $A\defeq R[[T]]\{T^{-1}\}$ and $f : \mathrm{Spf} \left(B \right) \rightarrow \mathrm{Spf} \left( A \right) $ a non-trivial Galois cover of degree $p$. 
We assume that $\pi$ (which is a uniformiser of $A$) is a unifomiser of $B$ (this condition is satisfied after possibly base changing to a finite extension of $R$, cf. [Epp]). Proposition 2.3 in [Saidi1] shows that $f$ has the structure of a torsor under one of the three group schemes, $\mu_p$, $\mathcal{H}_n$ where $0 < n < v_K(\lambda)$, or $\mathcal{H}_{ v_K(\lambda)}$ (cf. loc. cit. 2.1 for the definition of these group schemes and the local explicit description of torsors under these group schemes). 
To the torsor $f$ are associated some data: the acting group scheme as above, the \textbf{degree of different} $\delta$, the \textbf{conductor variable} m,
and $c=-m$ the \textbf{conductor}
(cf. loc. cit. definition 2.4. The notation $c$ is introduced in this paper, only the conductor variable $m$ was considered in loc. cit.). 
Adapting slightly the proof of Proposition 2.3 in [Sa\"\i di1] provides the following details in the three occurring cases:

\textbf{(a)} For the \textbf{group scheme} $\bm{ \mu_p}$ where $\delta=v_K(p)$, the torsor equation is of the form \[Z^p= u\] where $u = \sum_{i \in \mathbb{Z}} a_i T^i\in A^{\times}$ 
is a unit such that its image $\bar u$ modulo $\pi$ is not a $p$-power. 
On the level of special fibres the  induced $\mu_p$-torsor is given by an equation $z^p= \bar{u}$ where $\bar{u} =  \sum_{i \geq l} \bar{a}_i t^i \in k((t))$ for some integer $l$ with $\bar{a}_l \not= 0$
(here $t$ equals $T$ modulo $\pi$). 
There are two cases to consider:

\textbf{(a1)} $\mathrm{gcd}(l,p)=1$. We have $\bar u=t^l(\sum _{i \geq l} \bar{a}_i t^{i-l})$
and $\bar v\defeq \sum_{i \geq l} \bar{a}_i t^{i-l} \in k[[t]]$ is a unit. Further, we can write $u=T^lv$ where $v\defeq \sum_{i \geq l} {a}_i T^{i-l} 
\in R[[T]]\{T^{-1}\}$ is a unit whose reduction modulo $\pi$ equals $\bar v$.
After possibly multiplying $u$ by a $p$-power we can assume $0\le l<p$. The unit $v\in A^{\times}$ admits an $l$-th root $s\in A$ since $l$ is coprime to $p$ and $k$ is algebraically closed.
Thus, $s^l=v$ in $A$ and after replacing the parameter $T$ by $T'\defeq T.s$, which is also a parameter of $A$, our $\mu_p$-torsor $f : \mathrm{Spf} \left(B \right) \rightarrow \mathrm{Spf} 
\left( A \right)$ is defined by the equation $Z^p=(T')^l$.

\textbf{(a2)} $\mathrm{gcd}(l,p)>1$, in which case $l$ is divisible by $p$ 
and $\bar u =  \sum_{i \geq l} \bar{a}_i t^i$. 
After multiplying $u$ by $T^{-l}$ (which is a $p$-power) we can assume  $\bar u =  \sum_{i \geq l} \bar{a}_i t^{i-l}=\sum_ {j\geq 0}\bar a_jt^j\in k[[t]]$. 
Let $m := \mathrm{min} \{ i\ |\ v_K(a_i) = 0 , \mathrm{gcd}(i-l,p)=1 \} = \mathrm{min} \{ j \ |\ \mathrm{gcd}(j,p)=1 \}$. 
We can write $\bar u=\bar a_0+\bar a_1t^p+\cdots+\bar a_{[m/p]}t^{[m/p]p}+\bar a_mt^m+\text{higher order terms},$
and $u=a_0+a_1T^p+\cdots+a_{[m/p]}T^{[m/p]p}+\sum_{\substack{ v_K(a_i) = 0 \\ i \geq m }} a_i T^i +  \sum_{v_K(a_i) > 0} a_i T^i.$
If $a\in A$ is a unit we can write $a=b^p+c$ with $b\in A$ a unit and $v_K(c)>0$. Thus, we can assume without loss of generality that
$u=a_0^p+a_1^pT^p+\cdots+a_{[m/p]}^pT^{[m/p]p}+\sum_{\substack{ v_K(a_i) = 0 \\ i \geq m }} a_i T^i +  \sum_{v_K(a_i) > 0} a_i T^i.$
Now $a_0^p+a_1^pT^p+\cdots+a_{[m/p]}^pT^{[m/p]p}=(a_0+a_1T+\cdots+a_{[m/p]}T^{[m/p]})^p-p(a_0+a_1T+\cdots+a_{[m/p]}T^{[m/p]})+\text {higher order terms}$,
and after replacing $u$ by $u(a_0+a_1T+\cdots+a_{[m/p]}T^{[m/p]})^{-p}$ we can assume without loss of generality that the torsor equation is:
$Z^p = 1 + a_mT^m + \sum_{\substack{ v_K(a_i) = 0 \\ i > m }} a_i T^i +  \sum_{v_K(a_i) > 0} a_i T^i $.
Further, 
$a_mT^m + \sum_{\substack{ v_K(a_i) = 0 \\ i > m }} a_i T^i +  \sum_{v_K(a_i) > 0} a_i T^i=T^mv$
where $v=a_m + \sum_{\substack{ v_K(a_i) = 0 \\ i > m }} a_i T^{i-m} +  \sum_{v_K(a_i) > 0} a_i T^{i-m}\in A$ is a unit which admits an $m$-th root $u\in A$. 
Thus, $u^m=v$ in $A$ and after replacing the parameter $T$ by $T'\defeq T.u$, which is also a parameter of $A$, our $\mu_p$-torsor $f : \mathrm{Spf} \left(B \right) \rightarrow \mathrm{Spf} 
\left( A \right)$ is defined by the equation $Z^p=1+(T')^m$.

\textbf{Simplified form:} After a possible change of the parameter $T$ of $A$, the torsor equation $Z^p= u$ can be reduced to either the form \[\text{\textbf{(a1)}} \ \ \ Z^p = T^h \ \text {where}\ h \in \mathbb{F}_p^{\times},\] or of the form \[\text{\textbf{(a2)}} \ \ \ Z^p = 1 + T^m \] where $m$ is as defined above for these two cases. The conductor is given in both cases by
$\bf {(a1)} \ \ c=0$, and $\bf {(a2)} \ \ c=-m$.

\begin{center}
\line(1,0){250}
\end{center}

\textbf{(b)} For the \textbf{group scheme} $\bm{\mathcal{H}_n}$  where $0 < n < v_K(\lambda)$ and $\delta=v_K(p)-n(p-1)$, the torsor equation is of the form \[(1+ \pi^{n} Z)^p=1+ \pi^{np} u \] where $u = \sum_{i \in \mathbb{Z}} a_i T^i\in A^{\times}$  is a unit such that modulo $\pi$ it is not a $p$-power. 
Reducing modulo $\pi$, on the special fibre the acting group scheme is $\alpha_p$ and the torsor is given by an equation $z^p= \bar{u}$ where $\bar{u} =  \sum_{i \geq l} \bar{a}_i t^i \in k((t))$ for some integer $l$ with $\bar{a}_l \not= 0$ and which is defined up to addition of a $p$-power. 
We define $m := \mathrm{min} \{ i | v_K(a_i) = 0 , \mathrm{gcd}(i,p)=1 \} \in \mathbb{Z} .$
Then
$\bar u=\bar a_{l/p}t^{p.l/p}+\cdots+\bar a_{[m/p]}t^{[m/p]p}+\bar a_mt^m+\text{higher order terms},$
and $u=a_{l/p}T^{p.l/p}+\cdots+a_{[m/p]}T^{[m/p]p}+\sum_{\substack{ v_K(a_i) = 0 \\ i \geq m }} a_i T^i +  \sum_{v_K(a_i) > 0} a_i T^i.$
If $a\in A$ is a unit we can write $a=b^p+c$ with $b\in A$ a unit and $v_K(c)>0$. Thus, we can assume without loss of generality that
$u=a_{l/p}^pT^{p.l/p}+\cdots+a_{[m/p]}^pT^{[m/p]p}+\sum_{\substack{ v_K(a_i) = 0 \\ i \geq m }} a_i T^i +  \sum_{v_K(a_i) > 0} a_i T^i.$
Now $1+\pi^{np}(a_{l/p}^pT^{p.l/p}+\cdots+a_{[m/p]}^pT^{[m/p]p})=(1+\pi^n(a_{l/p}T^{l/p}+\cdots+a_{[m/p]}T^{[m/p]}))^p-p\pi^n(a_{l/p}T^{l/p}+\cdots+a_{[m/p]}T^{[m/p]})+\text {higher order terms}$. Thus, since
$np<v_K(p)+n$, the torsor equation can be written
 $(1+ \pi^{n} Z)^p=(1+\pi^n(a_{l/p}T^{l/p}+\cdots+a_{[m/p]}T^{[m/p]}))^p+\pi^{np}(\sum_{\substack{ v_K(a_i) = 0 \\ i \geq m }} a_i T^i +  \sum_{v_K(a_i) > 0} a_i T^i)$
and after multiplying by $(1+\pi^n(a_{l/p}T^{l/p}+\cdots+a_{[m/p]}T^{[m/p]}))^{-p}=1-p\pi^n(a_lT^{l/p}+\cdots+a_{[m/p]}T^{[m/p]})+\text {higher order terms}$, we get an equation
$(1+ \pi^{n} Z)^p=1+\pi^{np}(\sum_{\substack{ v_K(a_i) = 0 \\ i \geq m }} a_i T^i +  \sum_{v_K(a_i) > 0} a_i T^i)$
and can assume
$u =  \sum_{\substack{v_K(a_i) = 0 \\ i \geq m}} a_i T^i +  \sum_{v_K(a_i) > 0} a_i T^i.$
 Further, $a_mT^m + \sum_{\substack{ v_K(a_i) = 0 \\ i > m }} a_i T^i +  \sum_{v_K(a_i) > 0} a_i T^i=T^mv$ where $v\in A$ is a unit which admits an $m$-th root $h\in A$. 
Thus, $h^m=v$ in $A$ and after replacing the parameter $T$ by $T'\defeq T.h$, which is also a parameter of $A$, the
$\mathcal{H}_n$-torsor $f : \mathrm{Spf} \left(B \right) \rightarrow \mathrm{Spf} 
\left( A \right)$ is defined by the equation $Z^p=1+\pi^{np}(T')^m$. 

\textbf{Simplified form:} After a change of the parameter $T$, the torsor equation can be reduced to the form \[Z^p=1+ \pi^{np} T^m\] where $m$ is as defined above. The conductor is given by
$c=-m$.

\begin{center}
\line(1,0){250}
\end{center}

\textbf{(c)} For the \textbf{group scheme} $\bm{\mathcal{H}_ {v_K(\lambda)}}$  where $\delta=0$, the torsor equation is of the form \[(1+ \lambda Z)^p=1+ \lambda^{p} u \] where $u = \sum_{i \in \mathbb{Z}} a_i T^i\in A^{\times}$ is a unit. On the special fibre the acting group scheme is $\mathbb{Z}/p\mathbb{Z}$ and the torsor is given by an equation $z^p-z=\bar{u}$ where $\bar{u} =  \sum_{i \geq l} \bar{a}_i t^i$ for some integer $l$ with $\bar{a}_l \not= 0$ and which is defined up to addition of an Artin-Schreier element of the form $b^p - b$.
In fact, after such an Artin-Schreier transformation, $\bar{u}$ can be represented as: 
$\bar{u} = \bar{a}_m t^m + \bar{a}_{m+1} t^{m+1} + ... +  \bar{a}_{-1} t^{-1} = \sum_{i=m}^{-1} \bar{a}_i t^i $
where $\bar{a}_m \not=0$ and $m < 0$ is the conductor variable such that $\mathrm{gcd}(m,p)=1$. 
Indeed, for $f(t)=\sum_{i\ge 0}a_it^i\in k[[t]]$ we have $f(t)=(f(t)+f(t)^p+f(t)^{p^2}+\cdots)-(f(t)+f(t)^p+f(t)^{p^2}+\cdots)^p$.
Moreover, $\bar{u} = \bar{a}_m t^m + \bar{a}_{m+1} t^{m+1} + ... +  \bar{a}_{-1} t^{-1} = t^m \bar v$ where $\bar v=\bar{a}_m + \bar{a}_{m+1} t+ ... +  \bar{a}_{-1} t^{-m-1}\in k[[t]]$ is a unit.
Let $v={a}_m + {a}_{m+1} T+ ... +  {a}_{-1} T^{-m-1}\in R[[T]]$ be an element which lifts $\bar v$ and $h$ an $m$-th root of $v$ in $R[[T]]$.
Then after replacing the parameter $T$ by $T'\defeq T.h$, which is also a parameter of $A$, our 
$\mathcal{H}_{v_K(\lambda)}$-torsor $f : \mathrm{Spf} \left(B \right) \rightarrow \mathrm{Spf} 
\left( A \right)$ is defined by the equation $Z^p=1+\lambda^{p}(T')^m$.

\textbf{Simplified form:} After a change of the parameter $T$, the torsor equation over $R$ can be simplified to the form \[Z^p=1+ \lambda^{p} T^m\] where $m$ is as defined above. The conductor is given by $c=-m$.

\begin{center}
\line(1,0){250}
\end{center}


\section*{\S 1. The type $(p,p)$ case}
In this section $A\defeq R[[T]]\{ T^{-1} \}$ and $X_b \defeq \mathrm{Spf} \left( A \right)$. 
Let $f_{i, K} : (X_{i,b})_K \rightarrow (X_b)_K$ be two (generically) disjoint non-trivial degree $p$ Galois covers. 
We have the following diagram:
\begin{equation*}
 \xymatrix{
&& (Y_b)_K \defeq (X_{1,b})_K \times_{(X_b)_K} (X_{2,b})_K \ar@{->}[ddll]_{G'_{1,K}} \ar@{->}[ddrr]^{G'_{2,K}} \ar@{->}[dddd]_{G_{1,K} \times G_{2,K}} \\
\\
(X_{1,b})_K &&&& (X_{2,b})_K\\
\\
&& (X_b)_K   \ar@{<-}[uull]^{G_{1,K}} \ar@{<-}[uurr]_{G_{2,K}}
}
\end{equation*}
where $G_{i,K}$ and $G'_{i,K}$ are the acting group schemes on the various covers and as $\mathrm{char}(K)=0$, we have $G_{i,K} = G'_{i,K}\simeq \mathbb{Z}/p\mathbb{Z} \simeq \mu_p$  is \'etale for $i=1,2$.
For $i=1,2$, let $f_i :  X_{i,b} \rightarrow X_b$ be the Galois covers of degree $p$ where $X_{i,b}$ is the {\it normalisation} of $X_b$ in $(X_{i,b})_K$. 
Similarly, let $Y_b$ be the {\it normalisation} 
of $X_b$ in $\left( Y_{b} \right)_K$ so that $f : Y_b \rightarrow X_b$ is a non-trivial Galois cover of type $(p,p)$. We assume that $(Y_b)_k$ is {\it reduced}. 
Note that in this case $X_{i,b}$ is isomorphic to $\mathrm{Spf} \left( R[[T_i]]\{ T_i^{-1} \} \right)$ for $1\le i \le2$ (cf. [Bourbaki], \S2, 5]). 
We have the diagram: 
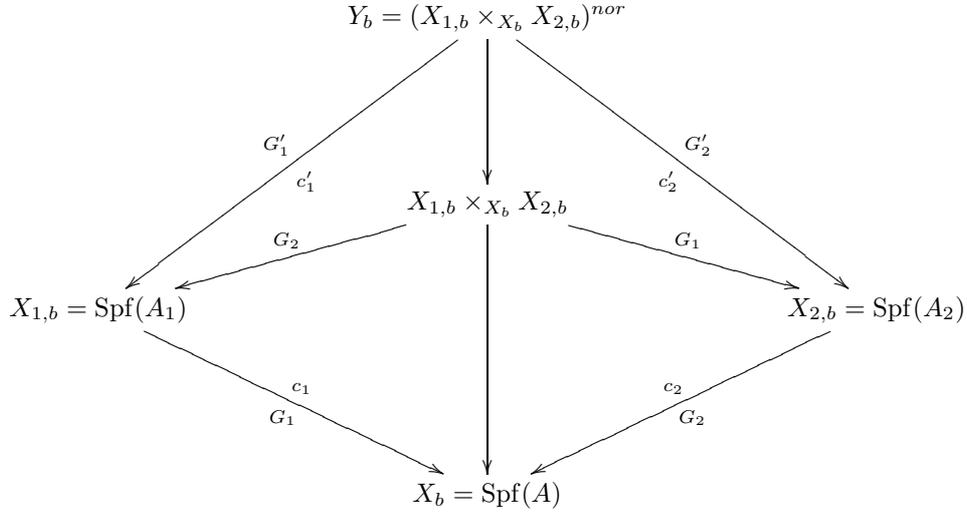
\begin{equation*}
 \xymatrix{
&& Y_b =({ X_{1,b} \times_{X_b} X_{2,b} })^{nor}  \ar@{->}[dddrr]_{c_2'}^{G'_2}  \ar@{->}[dddll]^{c_1'}_{G'_1}   \ar@{->}[dd] \\
\\
&& X_{1,b} \times_{X_b} X_{2,b}  \ar@{->}[dll]_{G_2} \ar@{->}[drr]^{G_1}  \ar@{->}[ddd]  \\
X_{1,b} = \mathrm{Spf} ( A_1 )  \ar@{->}[ddrr]^{c_1}_{G_1}  &&&& X_{2,b} = \mathrm{Spf} ( A_2 ) \ar@{->}[ddll]_{c_2}^{G_2} \\
\\
&& X_b= \mathrm{Spf} ( A  )
}
\end{equation*}
where:
\begin{itemize}
\item $Y_b$ and $X_{i,b}$ are normal for $i=1,2$, and $Y_b$ is the normalisation $({X_{1,b} \times_{X_b} X_{2,b} })^{nor}$ of the fibre product $({X_{1,b} \times_{X_b} X_{2,b} })$.
\item $c_i$ (respectively $c_i'$) denotes the conductor of the torsor $X_{i,b} \rightarrow X_b$ (respectively $Y_b  \rightarrow X_{i,b}$). The conductor $c_i$  (respectively $c_i'$) is dependent on the conductor variable $m_i$ (respectively $m_i'$) (cf. Background).
\item $ G_i$ (respectively $G'_i$) denotes the finite and flat (commutative) $R$-group scheme of the torsor $X_{i,b} \rightarrow X_b$ (respectively $Y_b  \rightarrow X_{i,b}$). We know $G_i$, $G'_i$ are among the $R$-group schemes $\mathcal{H}_{v_K(\lambda)}$, $\mu_p$, or $\mathcal{H}_{n}$ for $0 < n < v_K(\lambda)$ (cf. loc. cit.).
\end{itemize}

On the level of special fibres over $k$ we have a diagram:
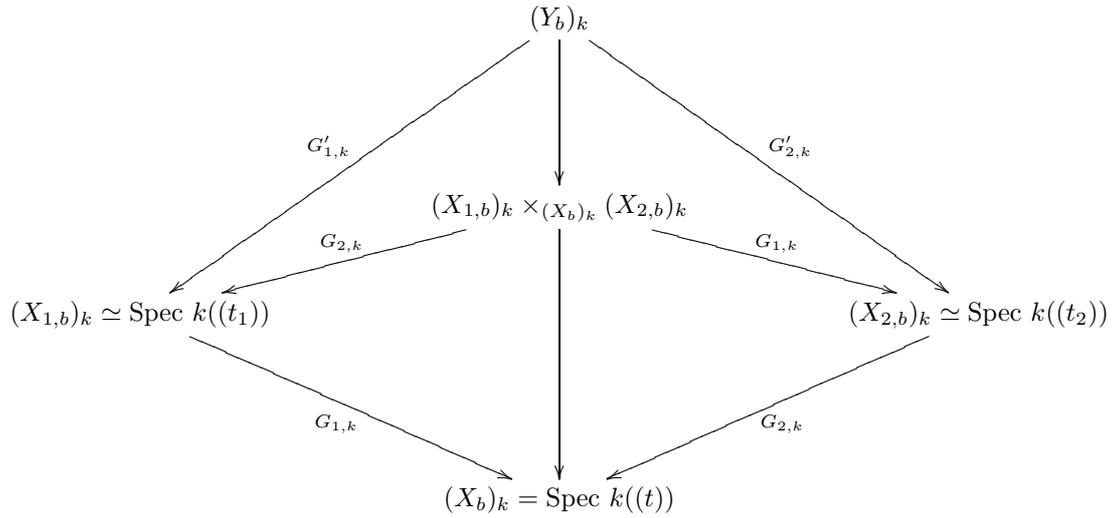
\begin{equation*}
 \xymatrix{
&& (Y_b)_k  \ar@{->}[dddrr]^{G'_{2,k}}  \ar@{->}[dddll]_{G'_{1,k}}   \ar@{->}[dd] \\
\\
&& (X_{1,b})_k \times_{(X_b)_k}  (X_{2,b})_k \ar@{->}[dll]_{G_{2,k}} \ar@{->}[drr]^{G_{1,k}}  \ar@{->}[ddd]  \\
(X_{1,b})_k\simeq \mathrm{Spec}\ k((t_1)) \ar@{->}[ddrr]_{G_{1,k}}  &&&& (X_{2,b})_k\simeq \mathrm{Spec} \ k((t_2)) \ar@{->}[ddll]^{G_{2,k}} \\
\\
&& (X_b)_k =  \mathrm{Spec} \ k((t))  
}
\end{equation*}
where:
\begin{itemize}
\item $(Y_b)_k$ and $(X_{i,b})_k$ are reduced for $i=1,2$.
\item  $k((t)) =A / ( \pi )$ (respectively $A_1/\pi\simeq k((t_1))$, $A_2/\pi \simeq k((t_2))$)
where $t$ (respectively $t_1$ and $t_2$) is the reduction modulo $\pi$ of $T$ (respectively $T_1$, $T_2$, where $T_i$ is some suitable parameter of $X_{i,b}$ for $i=1,2$).
\item $G_{i,k}=G_i \times_R k$ and $G'_{i,k}=G'_i \times_R k$ are the acting group schemes over $k$, a field with characteristic $p$, so that these group schemes are necessarily isomorphic to either $\mathbb{Z}/p\mathbb{Z}$, $\mu_p$ or $\alpha_p$.
\end{itemize}

We aim to express the \textbf{conductor} $c'_1$ in terms of $c_1$ and $c_2$ for the various torsor combinations and likewise for $c'_2$. To achieve this, we express the \textbf{conductor variables} $m_1'$ and $m'_2$ in terms of $m_1$ and $m_2$. We have six cases to consider by taking all possible pairs of the group schemes $\mathcal{H}_{v_K(\lambda)}$, $\mu_p$ and $\mathcal{H}_{n}$ over $R$ acting on $X_{1,b}$, $X_{2,b}$. 
The following is one of our main results. 

\begin{thmA} Let $X_b = \mathrm{Spf} \left( A \right)$ and suppose we have two (generically) disjoint non-trivial degree $p$ Galois covers $f_{i, K} : (X_{i,b})_K \rightarrow (X_b)_K$, 
for $i=1,2$. Let $(Y_b)_K$ be the compositum of these covers. 

For $i=1,2$, let $f_i :  X_{i,b} \rightarrow X_b$ be the Galois covers of degree $p$ where $X_{i,b}$ is the normalisation of $X_b$ in $(X_{i,b})_K$.
Set $Y_b$ as the normalisation of $X_b$ in $\left( Y_{b} \right)_K$ so that $f : Y_b \rightarrow X_b$ is a non-trivial Galois cover of type $(p,p)$.  We assume that the ramification index of the corresponding extension of DVR's equals 1 and that the special fibre of $Y_b$ is reduced.
Thus, $f_i$ is a non-trivial torsor under a finite flat $R$-group scheme $G_i$ of rank $p$ with conductor variable $m_i$ for $i=1,2$. 
Let $m'_i$ denote the conductor variable of the torsor $Y_b \rightarrow X_{i,b}$. Then, for all possible pairs of $G_1$ and $G_2$, we can express the conductors $m'_i$ in terms of the $m_i$ conductor variables for $i=1,2$ as follows:

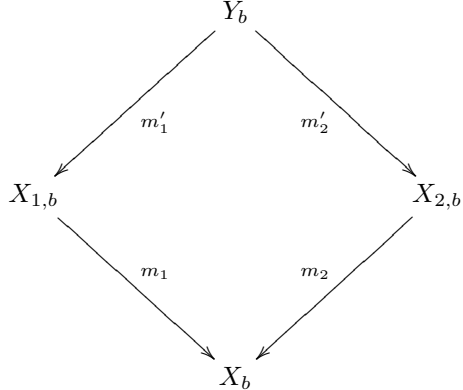
\begin{equation*}
 \xymatrix{
&& Y_b \ar@{->}[ddll]^{ m_1'}  \ar@{->}[ddrr]_{m_2'} \\
\\
X_{1,b} &&&& X_{2,b} \\
\\
&& X_b  \ar@{<-}[uull]_{m_1} \ar@{<-}[uurr]^{m_2}
}
\end{equation*}

\begin{enumerate}
\item For $G_1 = G_2= \mathcal{H}_{v_K(\lambda)}$ we have that $m'_1=m_2$ and $m'_2=m_1p-m_2(p-1)$ when $m_1 \leq m_2$, and $m'_1=m_2p-m_1(p-1)$ and $m'_2=m_1$ when $m_1 > m_2$.
\item For $G_1 = \mathcal{H}_{v_K(\lambda)} $ and $G_2=\mu_p$ we have that $m'_1=m_2p-m_1(p-1)$ and $m'_2=m_1$.
\item For $G_1 =\mathcal{H}_{v_K(\lambda)} $ and $G_2=\mathcal{H}_{n}$ we have that $m'_1=m_2p-m_1(p-1)$ and $m'_2=m_1$.
\item For $G_1 =\mathcal{H}_{n}$ and $G_2=\mu_p$ we have that $m'_1=m_2p-m_1(p-1)$ and $m'_2=m_1$.
\item For $G_1 =G_2=\mu_p $ we have that $m'_1=m_2$ and $m'_2=m_1p-m_2(p-1)$ when $m_1 \leq m_2$, and $m'_1=m_2p-m_1(p-1)$ and $m'_2=m_1 $ when $m_1 > m_2$. In this case these results are only valid when at least one of $m_1$ and $m_2$ is non-zero (cf. Proof and Remark 1.8).  
\item For $G_1 =\mathcal{H}_{n_1}$ and $G_2=\mathcal{H}_{n_2}$ we have that $m'_1 = m_2$ and $m'_2 = m_1p-m_2(p-1)$ when $n_1 < n_2$, that $m'_1 = m_2p-m_1(p-1)$ and $m'_2 = m_1$ when $n_1 > n_2$, that $m'_1 = m_2$ and $m'_2 = m_1p-m_2(p-1)$ when both $n_1 = n_2$ and $m_1 < m_2$, and $m'_1 =m_2p-m_1(p-1) $ and $m'_2 =m_1 $ when both $n_1 = n_2$ and $m_1 \geq m_2$, where $0 < n, n_1, n_2 < v_K(\lambda)$.
\end{enumerate}

\end{thmA}

\section*{\S 1.1 Proof of Theorem 1.1}

\begin{proof} We treat each of the six occurring cases individually. However, there is an important distinction between the first three cases and the remaining cases.

In the first three cases, that is when at least one of the acting group schemes is the \'{e}tale group scheme $\mathcal{H}_{v_K(\lambda)}$, one can work modulo $\pi$ at the level of special fibres for, in this case, $Y_b = X_{1,b} \times_{X_b} X_{2,b}$ which implies $ (Y_b)_k = (X_{1,b})_k  \times_{(X_b)_k} (X_{2,b})_k$. Indeed, suppose $G_1 = \mathcal{H}_{v_K(\lambda)}$ so that the torsor $X_{1,b} \rightarrow {X_b}$  is \'{e}tale. Then, by base change, the torsor $X_{1,b} \times X_{2,b} \rightarrow X_{2,b} $ is automatically \'{e}tale. The special fibre of $X_{2,b}$ is reduced (because it is dominated by $Y_b$ whose special fibre is reduced) but as $X_{1,b} \times X_{2,b} \rightarrow X_{2,b} $ is \'{e}tale, this implies the special fibre of $X_{1,b} \times X_{2,b}$ is also reduced. Then, by Theorem 3.4 in this paper, $X_{1,b} \times_{X_b} X_{2,b}$ is normal and equal to $Y_b$, as required. 

In the last three cases, we do not have this situation, which means one must work above $X_b$ over $R$ without being permitted to reduce to the special fibre. However, we still proceed in a similar fashion, even if the computations are more involved. In particular, we start with the equation of $X_{i,b}\to X_b$, base change it to $X_{j,b}$ for $j\neq i$
and make appropriate (Kummer) transformations in order to find the torsor equations of $Y_b \rightarrow X_{j,b}$ and read off the conductors $m'_j$ for $j=1,2$. 
Note that in each case we can perform a change of the parameter $T$ of $A= R[[T]]\{T^{-1}\}$ so that one of the two torsor equations above $X_b = \mathrm{Spf} \left( A \right)$ is in its simplified form 
but we must assume the other equation remains in its original full power series form (cf. Background).

\textbf{1.} $(\mathcal{H}_{v_K(\lambda)}, \mathcal{H}_{v_K(\lambda)})$.
Here $m_1, m_2 < 0$. The $\mathcal{H}_{v_K(\lambda)}$ torsor equation $X_{i,b}\to X_b$ is given by $(1+\lambda Z_i)^p = 1+\lambda^p u_i$ 
where $u_i \in  A^{\times}$. Modulo $\pi$, these torsor equations reduce to $z_i^p-z_i = \bar{u}_i$ on the special fibre, $i=1,2$. 
We start by computing $m'_1$. We can choose the parameter $T$ so that $u_1 = T^{m_1}$, $u_2 = \sum_{i \in \mathbb{Z}} a_i T^i$ is a power series, accordingly, $\bar{u}_1=t^{m_1}$ and $\bar{u}_2 = \sum_{i=m_2}^{-1} \bar{a}_i t^i$ (cf. loc. cit.) where $\bar{a}_{m_2} \not= 0$. We can write $t$ in terms of $z_1$ in $(X_{1,b})_k$:
$z_1^p-z_1 = t^{m_1}\Leftrightarrow z_1^p \left( 1-z_1^{1-p} \right) = t^{m_1} \Leftrightarrow t =   \left( z_1^{1/m_1} \right)^p \left( 1-z_1^{1-p} \right)^{1/m_1}.$
Thus a parameter of $(X_{1,b})_k$ is $z_1^{1/m_1}$ and so by letting $z:= z_1^{1/m_1}$ we can write 
$t =  z^p \left( 1-z^{-m_1(p-1)} \right)^{1/m_1}$.
We can now proceed to base change the torsor equation of $(X_{2,b})_k \rightarrow (X_b)_k$ to $(X_{1,b})_k$ to obtain the torsor equation for $(Y_b)_k \rightarrow (X_{1,b})_k$:
$$z_2^p-z_2 = \sum_{i=m_2}^{-1} \bar{a}_i t^i =  \sum_{i=m_2}^{-1} \bar{a}_i  z^{ip} \left( 1-z^{-m_1(p-1)} \right)^{i/m_1} 
=  \sum_{i=m_2}^{-1} \bar{a}_i z^{ip} \left( 1- \frac{i}{m_1} z^{-m_1(p-1)} + ... \right) $$
$$=  \bar{a}_{m_2} z^{m_2p} \left( 1- \frac{m_2}{m_1} z^{-m_1(p-1)} + ... \right) +  \bar{a}_{m_2+1} z^{(m_2+1)p} \left( 1- \frac{m_2+1}{m_1} z^{-m_1(p-1)} + ... \right)+...$$
$=  \bar{a}_{m_2} z^{m_2p} - \frac{m_2 \bar{a}_{m_2} }{m_1} z^{m_2p-m_1(p-1)} + \text{higher order terms}.$

Expressing  $z^{ip}$ as $z^{ip}-z^{i}+z^{i}$ gives rise (after an Artin-Schreier transformation) to an equation of the form:
$z_2^p-z_2 = \bar{a}_{m_2} z^{m_2} - \frac{m_2 \bar{a}_{m_2} }{m_1} z^{m_2p-m_1(p-1)} +   \text{higher order terms}$. 
The conductor variable $m'_1$ is the smallest power of $z$ in the above expression which is coprime to $p$. The expression above indicates there are two candidates, namely $m_2$ and $m_2p-m_1(p-1)$.  Note that $m_2p - m_1(p-1) \leq m_2$ is equivalent to $m_1 \geq  m_2$. Therefore, when $m_1 \geq  m_2$ we have $m'_1 = m_2p - m_1(p-1)$ and when $m_1 < m_2$ we have $m'_1= m_2$.  The formula for $m'_2$ is obtained in a similar way as a consequence of the symmetry occurring in this case.


\begin{center}
\line(1,0){250}
\end{center}

\textbf{2.} $(\mathcal{H}_{v_K(\lambda)}, \mu_p)$.
Here $m_1 < 0$ while $m_2 \geq 0$ hence $m_1 \leq m_2$. The torsor equation for $\mathcal{H}_{v_K(\lambda)}$ is given by $(1+\lambda Z_1)^p = 1+\lambda^p u_1$ and for $\mu_p$ by $Z_2^p = u_2$ where $u_1, u_2 \in A^{\times}$. Modulo $\pi$, these torsor equations reduce to $z_1^p-z_1 = \bar{u}_1$ and  $z_2^p= \bar{u}_2$ with acting group schemes $\mathbb{Z}/p\mathbb{Z}$ and $\mu_p$ respectively on the special fibre.

We start by computing $m'_1$. We can choose the parameter $T$ so that $u_1 = T^{m_1}$ but $u_2 = \sum_{i \in \mathbb{Z}} a_i T^i$ remains as a power series and therefore, accordingly, $\bar{u}_1=t^{m_1}$ and $\bar{u}_2 = \sum_{i  \geq l} \bar{a}_i t^i$ for some integer $l$ where $\bar{a}_{l} \not= 0$. As in case 1 of this proof, we have that the parameter of $(X_{1,b})_k$ is $z:=z_1^{1/m_1}$ and we can write $t = z^p \left( 1-z^{-m_1(p-1)} \right)^{1/m_1}$. We now have two cases to treat, namely $(a1)$ and $(a2)$, depending on whether or not $l$ is coprime to $p$.

\textbf{(a1)}  In this case, $\mathrm{gcd}(l,p)=1 \Rightarrow m_2=0$. We base change the torsor equation of $(X_{2,b})_k \rightarrow (X_b)_k$ to $(X_{1,b})_K$
to obtain the torsor equation for $ (Y_b)_k \rightarrow (X_{1,b})_k$:
$$z_2^p   =  \sum_{i  \geq l} \bar{a}_i t^i = \sum_{i  \geq l} \bar{a}_i z^{ip} (1-z^{-m_1(p-1)})^{i/m_1} 
 = \sum_{i  \geq l} \bar{a}_i z^{ip} (1- \frac{i}{m_1}z^{-m_1(p-1)}+...) $$
$$= \bar{a}_{l} z^{lp} ( 1- \frac{l}{m_1}z^{-m_1(p-1)}+...) + \bar{a}_{l+1} z^{(l+1)p} (1- \frac{l+1}{m_1}z^{-m_1(p-1)}+...)+... $$
As this is a $\mu_p$-torsor equation, the factor $\bar{a}_{l} z^{lp}$ can be eliminated by multiplication by a suitable $p$-power to obtain an equation:
$z_2^p = 1- \frac{l}{m_1}z^{-m_1(p-1)}+  \text{higher order terms}$.
So the conductor variable is $m'_1=-m_1(p-1)$, as this is the smallest power of $z$ in the above expression which is coprime to $p$. 

\textbf{(a2)}  In this case $\mathrm{gcd}(l,p) \not=1$. By the details outlined at the start of this paper (cf. Background), 
we know that the torsor equation of $(X_{2,b})_k \rightarrow (X_b)_k$ can be expressed as follows:
$z_2^p   =  1+ \bar{a}_{m_2}t^{m_2} + \sum_{i  > m_2} \bar{a}_{i} t^{i} =  1+ \sum_{i  \geq m_2} \bar{a}_{i} t^{i}.$
By a suitable change of variables, we can express this $\mu_p$-torsor as:
$z_2^p  =  1+ \sum_{i  \geq m_2} \bar{a}_{i} t^{i}  \Leftrightarrow (z_2 - 1)^p= \sum_{i  \geq m_2} \bar{a}_{i} t^{i} \Rightarrow z_2^p  = \sum_{i  \geq m_2} \bar{a}_{i} t^{i}.$
We can now proceed to base change the torsor equation of $(X_{2,b})_k \rightarrow (X_b)_k$ to $(X_{1,b})_k$ to obtain the torsor equation for $ (Y_b)_k \rightarrow (X_{1,b})_k$:
$$z_2^p   =  \sum_{i  \geq m_2} \bar{a}_{i} z^{pi} \left( 1-z^{-m_1(p-1)} \right)^{i/m_1} 
=  \sum_{i  \geq m_2} \bar{a}_{i} z^{pi} \left( 1 - \frac{i}{m_1}z^{-m_1(p-1)} + ... \right) $$
$$=  \bar{a}_{m_2} z^{m_2p} \left( 1 - \frac{m_2}{m_1}z^{-m_1(p-1)} + ... \right) + \bar{a}_{m_2+1} z^{(m_2+1)p} \left( 1 - \frac{m_2+1}{m_1}z^{-m_1(p-1)} + ... \right)+... $$
$=  \bar{a}_{m_2} z^{m_2p} - \frac{m_2 \bar{a}_{m_2} }{m_1}z^{m_2p-m_1(p-1)} +  \text{higher order terms}.$

So $m'_1=m_2 p-m_1(p-1)$ as this is the smallest power of $z$ in the above expression which is not divisible by $p$. 

We now determine $m'_2$. We choose $T$ so that $u_1 = \sum_{i \in \mathbb{Z} } a_i T^i$ and $u_2$ is given by $T^h$ in the case (a1) and by $1 +T^{m_2}$ in the case (a2). After reducing these equations modulo $\pi$, we have $z_1^p-z_1=\sum_{i=m_1}^{-1}  \bar{a}_i t^i$ and (a1) $z_2^p = t^{h}$ or (a2) $z_2^p =1+ t^{m_2}$ on the special fibre. 

\textbf{(a1)} We can write $t$ in terms of $z_2$ in $(X_{2,b})_k$ since $z_2^p =t^{h} \Leftrightarrow t= \left( z_2^{1/h} \right)^p$. This implies that $z :=  z_2^{1/h}$ is a parameter of $(X_{2,b})_k$ 
and we have that $t=z^p$. We base change the equation of $(X_{1,b})_k\to (X_b)_k$ to $(X_{2,b})_k$
to obtain the torsor equation for $ (Y_b)_k \rightarrow (X_{2,b})_k$:
$z_1^p - z_1  =  \sum_{i=m_1}^{-1}  \bar{a}_i t^i  = \sum_{i=m_1}^{-1}  \bar{a}_i z^{ip}  
= \bar{a}_{m_1} z^{m_1p} +  \text{higher order terms}$. 
The leading term $z^{m_1p} $ (as well as all the other terms $z^{ip}$) is a multiple of $p$ but, as in case 1 of this proof, after an Artin-Schreier transformation we obtain:
$z_1^p - z_1 =  \bar{a}_{m_1} z^{m_1} +  \text{higher order terms}$.  
Therefore, the conductor variable $m'_2=m_1$.

\textbf{(a2)} As above, we write $t$ in terms of $z_2$ in $(X_{2,b})_k$:
$z_2^p = 1+t^{m_2}  \Leftrightarrow z_2^p -1= t^{m_2} 
\Leftrightarrow   (z_2 -1 )^p= t^{m_2}  
\Leftrightarrow t=  \left( (z_2 -1 )^{1/m_2}  \right)^p.$
This means that the parameter of $(X_{2,b})_k$ is $z:=(z_2 -1 )^{1/m_2}$ and so, from the above, we obtain $t =  z^p$. Now, we base change 
the equation of $(X_{1,b})_k\to (X_b)_k$ to $(X_{2,b})_k$
to obtain the torsor equation for $ (Y_b)_k \rightarrow (X_{2,b})_k$:
$z_1^p - z_1 =  \sum_{i=m_1}^{-1}  \bar{a}_i t^i = \sum_{i=m_1}^{-1}  \bar{a}_i z^{ip}  
= \bar{a}_{m_1} z^{m_1p} + \text{higher order terms}.$ 
After an Artin-Schreier transformation we obtain:
$z_1^p - z_1 =  \bar{a}_{m_1} z^{m_1}  + \text{higher order terms}.$ 
Therefore, as in the (a1) case, the conductor variable $m'_2=m_1$.

\begin{center}
\line(1,0){250}
\end{center}

\textbf{3.} $(\mathcal{H}_{v_K(\lambda)}, \mathcal{H}_n)$.
Here $m_1 < 0$ while $m_2 \in \mathbb{Z}$.  The torsor equation for $\mathcal{H}_{v_K(\lambda)}$ is given by $(1+\lambda Z_1)^p = 1+\lambda^p u_1$ and for $\mathcal{H}_n$ by $(1+\pi^n Z_2)^p = 1+ \pi^{np} u_2$ where $u_1, u_2 \in A^{\times}$. Modulo $\pi$, these torsor equations reduce to $z_1^p-z_1 = \bar{u}_1$ and  $z_2^p= \bar{u}_2$ with acting group schemes $\mathbb{Z}/p\mathbb{Z}$ and $\alpha_p$ respectively on the special fibre.

We start by computing $m'_1$. We can choose the parameter $T$ so that $u_1 = T^{m_1}$ but $u_2 = \sum_{i \in \mathbb{Z} } a_i T^i$ remains as a power series and therefore, accordingly, $\bar{u}_1=t^{m_1}$ and $\bar{u}_2 = \sum_{i  \geq l} \bar{a}_i t^i$ for some integer $l$ where $\bar{a}_{l} \not= 0$. Recall that $l=m_2$ here. As in case 1 of this proof, we have that the parameter of $(X_{1,b})_k$ is $z:=z_1^{1/m_1}$ and we can write $t = z^p \left( 1-z^{-m_1(p-1)} \right)^{1/m_1}$. We base change the torsor equation of $(X_{2,b})_k \rightarrow (X_b)_k$ to $(X_{1,b})_k$ to obtain the torsor equation for $ (Y_b)_k \rightarrow (X_{1,b})_k$:
$$z_2^p  =  \sum_{i  \geq l} \bar{a}_i t^i =  \sum_{i  \geq l} \bar{a}_i z^{pi} \left( 1-z^{-m_1(p-1)} \right)^{i/m_1}   
=  \sum_{i  \geq l} \bar{a}_i z^{pi} \left( 1-\frac{i}{m_1} z^{-m_1(p-1)} +... \right)$$ 
$$=  \bar{a}_l z^{pl} \left( 1-\frac{l}{m_1} z^{-m_1(p-1)} +... \right) +\bar{a}_{l+1} z^{p(l+1)} \left( 1-\frac{l+1}{m_1} z^{-m_1(p-1)} +... \right)+...$$
$=  \bar{a}_l z^{pl} -\frac{ l \bar{a}_l}{m_1} z^{lp-m_1(p-1)} + \text{higher order terms}.$

As this is an $\alpha_p$-torsor equation, the term $\bar a_lz^{pl}$ can be removed and we can ignore the terms involving $i$'s
which are divisible by $p$. So the conductor variable $m'_1=m_2 p-m_1(p-1)$, as this would be the smallest power of $z$ which is not divisible by $p$.

It remains to compute $m'_2$ in this case. This time we choose the parameter $T$ so that $u_1 = \sum_{i \in \mathbb{Z} } a_i T^i$ is the power series and $u_2 =  T^{m_2}$. After reducing modulo $\pi$, we have $\bar{u}_1=\sum_{i=m_1}^{-1}  \bar{a}_i t^i$ and $\bar{u}_2 = t^{m_2}$ on the special fibre. We can write $t$ in terms of $z_2$ in $(X_{2,b})_k$ since $z_2^p =t^{m_2} \Leftrightarrow t= \left( z_2^{1/m_2} \right)^p$. This implies that $z :=  z_2^{1/m_2}$ is the parameter of $(X_{2,b})_k$ and we have that $t=z^p$. We base change 
the equation of $(X_{1,b})_k \rightarrow (X_b)_k$ to $(X_{2,b})_k$ 
to obtain the torsor equation for $ (Y_b)_k \rightarrow (X_{2,b})_k$:
$z_1^p - z_1 =  \sum_{i=m_1}^{-1}  \bar{a}_i t^i  = \sum_{i=m_1}^{-1}  \bar{a}_i z^{ip}  = \bar{a}_{m_1} z^{m_1p} +  \text{higher order terms}.$
After an Artin-Schreier transformation we obtain:
$z_1^p - z_1 =  \bar{a}_{m_1} z^{m_1} +  \text{higher order terms}.$ 
Therefore, the conductor variable $m'_2=m_1$.

\begin{center}
\line(1,0){250}
\end{center}

We remind the reader that in the remaining three cases, we cannot reduce modulo $\pi$ and work at the level of special fibres. Thus, the computations here are slightly more involved. It will be useful to recall in advance here the following equality given by the Binomial Theorem
$$1 +  \left( \pi^n b Z \right)^p = \left( 1 +  \pi^n b Z \right)^p - \sum_{k=1}^{p-1} \binom{p}{k} \left( \pi^n b Z \right)^k \ \ \ (*)$$
which can be generalised to the Multimonomial Theorem (or Identity)
$$1 +  \sum_i \left( \pi^n b_i Z^i \right)^p = \left( 1 +  \sum_i \pi^n b_i Z^i \right)^p - p \sum_{i}  \pi^n b_i Z^i  + \text{higher order terms}\ \ \ (**).$$
We also mention here that to circumvent our inability to take, say, $p$-th roots of coefficients belonging to the ring $R$, 
we can adopt the following technique for a given element $a_i$ where $v_K(a_i)=0$; namely, it can be expressed as $a_i = b_i^p + c_i$ for some $b_i, c_i \in R$ such that $v_K(b_i)=0$ and $v_K(c_i)>0$.

\textbf{4.} $(\mathcal{H}_n, \mu_p)$.
Here $m_1 \in \mathbb{Z}$ while $m_2 \geq 0$.  The torsor equation for $\mathcal{H}_n$ is given by  $(1+\pi^n Z_1)^p = 1+ \pi^{np} u_1$ and for $\mu_p$ by $Z_2^p = u_2$ where $u_1, u_2 \in A^{\times}$. 
We start by computing $m'_1$. We can choose the parameter $T$ so that $u_1 = T^{m_1}$ but $u_2 = \sum_{i \in \mathbb{Z}} a_i T^i$ remains as a power series. We can express $T$ in terms of $Z_1$ in order to read off the parameter for $X_{1,b}$:
$(1+\pi^n Z_1)^p = 1+ \pi^{np} T^{m_1}  \Leftrightarrow  \pi^{np} Z_1^p + \sum_{k=1}^{p-1} \binom{p}{k} \pi^{nk} Z_1^k +1 = 1+ \pi^{np} T^{m_1}
\Leftrightarrow  \pi^{np} Z_1^p + \sum_{k=1}^{p-1} \binom{p}{k} \pi^{nk} Z_1^k = \pi^{np} T^{m_1}
\Leftrightarrow Z_1^p \left(  \pi^{np}  + \sum_{k=1}^{p-1} \binom{p}{k} \pi^{nk} Z_1^{k-p}  \right) = \pi^{np} T^{m_1}  
\Leftrightarrow Z_1^p \left(  1  + \sum_{k=1}^{p-1} \binom{p}{k} \pi^{n(k-p)} Z_1^{k-p}  \right) =  T^{m_1}  
\Leftrightarrow T=\left( Z_1^{1/m_1} \right)^p \left(  1  + \sum_{k=1}^{p-1} \binom{p}{k} \pi^{n(k-p)} Z_1^{k-p}  \right)^{1/m_1}.$

We know from the proof of case 3 that $Z:= Z_1^{1/m_1}$ is the parameter of $X_{1,b}$ modulo $\pi$ (hence it is a parameter of $X_{1,b}$) and so we can write:
$$T = Z^p \left(  1  + \sum_{k=1}^{p-1} \binom{p}{k} \pi^{-n(p-k)} Z^{-m_1(p-k)}  \right)^{1/m_1}. $$
For convenience, set $B=  \sum_{k=1}^{p-1} \binom{p}{k} \pi^{-n(p-k)} Z^{-m_1(p-k)}$ so that $T= Z^p \left(  1  + B \right)^{1/m_1}$. We now have two cases to treat, namely $(a1)$ and $(a2)$, depending on whether or not $l=\mathrm{min} \{ i | v_K(a_i) = 0 \}$ is coprime to $p$.

\textbf{(a1)}  In this case, $\mathrm{gcd}(l,p)=1 \Rightarrow m_2=0$. We base change the torsor equation of $X_{2,b} \rightarrow X_b$ to $X_{1,b}$ to obtain:
$Z_2^p  = \sum_{i \in \mathbb{Z}} a_i T^i = \sum_{i \in \mathbb{Z}} a_i  Z^{ip} \left(  1  + B \right)^{i/m_1}= \sum_{i \in \mathbb{Z}} a_i  Z^{ip} \left(  1  + \frac{i}{m_1} B +... \right)
= \sum_{i \in \mathbb{Z}} a_i  Z^{ip}   + \sum_{i \in \mathbb{Z}} \frac{i a_i}{m_1}  Z^{ip} B +...
= \sum_{v_K(a_i)=0} a_i  Z^{ip}  +  \sum_{v_K(a_i)>0} a_i  Z^{ip} + \sum_{i \in \mathbb{Z}} \frac{i a_i}{m_1}  Z^{ip} B +...$.
For the terms where $v_K(a_i)=0$, we can express $a_i = b_i^p + c_i$ for some $b_i, c_i \in R$ with $v_K(b_i)=0$ and $v_K(c_i)>0$ to obtain:
$Z_2^p = \sum_{v_K(b_i)=0} b_i^p  Z^{ip}  + \sum_{v_K(c_i)>0} c_i  Z^{ip} +  \sum_{v_K(a_i)>0} a_i  Z^{ip}  + \sum_{i \in \mathbb{Z}} \frac{i a_i}{m_1}  Z^{ip} B+...  
= \sum_{v_K(b_i)=0} \left( b_i  Z^i \right)^p  + \sum_{v_K(d_i)>0} d_i  Z^{ip} + \sum_{i \in \mathbb{Z}} \frac{i a_i}{m_1}  Z^{ip} B+... ,$
where we set $d_i = c_i $ if $v_K(a_i)=0$ and $d_i= a_i$ if $v_K(a_i)>0$. Now, as this is generically a $\mu_p$-torsor we can take the $p$-power term $\left( b_l Z^l \right)^p$ in the first summation
into factor, so that we get a new equation:
$Z_2^p  = 1 + \sum_{ \substack{ v_K(b_i)=0  \\  i \not= l   }} \left(  b_l^{-1} b_i   Z^{i-l} \right)^p + \sum_{v_K(d_i)>0} b_l^{-p} d_i  Z^{p(i-l) } +   \sum_{i \in \mathbb{Z}} \frac{i b_l^{-p}  a_i}{m_1} Z^{p(i-l)} B + ... ,$
which can be rewritten using the identity $(**)$, and after multiplying by a suitable $p$-power, as:
$Z_2^p = 1 - p \sum_{ \substack{ v_K(b_i)=0  \\  i \not= l   }}  b_l^{-1} b_i  Z^{i-l} + \sum_{v_K(d_i)>0} b_l^{-p} d_i  Z^{p(i-l) }  
+\sum_{i \in \mathbb{Z}} \frac{i b_l^{-p}  a_i}{m_1} Z^{p(i-l)} B + \text{higher order terms}. $ 
The summation $\sum_{v_K(d_i)>0} d_i b_l^{-p} Z^{p(i-l) }$ does not contribute to the conductor variable since the powers of $Z$ involved are $p$-powers so we can safely exclude it. 
Indeed, if the coefficient with smallest $K$-valuation in the right hand side of the above equation occurs in the above summation say in the term $d_i b_l^{-p} Z^{p(i-l)}$
 then $v_K(d_i)$ is necessarily divisible by $p$ (since $(Y_b)_k$ is reduced), and we can assume without loss of generality that this summation is of the form $\pi^{pt} (f(Z)^{p}+\pi g(Z))$,
 where $f(Z)\in A_1$ is a unit and $g(Z)\in A_1$. Writing $1+\pi ^{pt}f(Z)^{p}=(1+\pi^tf(Z))^p-\sum_{k=1}^{p-1} \binom{p}{k} \left( \pi^t f(Z)\right)^k$ and multiplying the above equation by 
 $(1+\pi^tf(Z))^{-p}$ we obtain an equation: 
$Z_2^p  = 1 - p \sum_{ \substack{ v_K(b_i)=0  \\  i \not= l   }}  b_l^{-1} b_i  Z^{i-l} - \sum_{k=1}^{p-1} \binom{p}{k} \left( \pi^t f(Z)\right)^k
+\sum_{i \in \mathbb{Z}} \frac{i b_l^{-p}  a_i}{m_1} Z^{p(i-l)} B + \text{higher order terms}$.
Furthermore the summation $\sum_{k=1}^{p-1} \binom{p}{k} \left( \pi^t f(Z)\right)^k$ doesn't contribute anymore towards the
coefficient with smallest possible valuation in the right hand side of the above equation. 
For the rest of this proof we will automatically operate in this way and ignore such summations. 
Then, up to multiplying the coefficients by units, we have:
$Z_2^p= 1 - \pi^{v_K(p)} \sum_{\substack { v_K(b_i)=0  \\  i \not= l   }}  b_l^{-1} b_i  Z^{i-l} 
 +  \pi^{v_K(p)-n(p-1)}  \sum_{i \in \mathbb{Z}} \frac{i b_l^{-p}  a_i}{m_1}  Z^{p(i-l)-m_1(p-1)} + \text{higher order terms}.$

Clearly the smallest power of $\pi$ is $v_K(p)-n(p-1)$ and so we look to the summation with that coefficient for the conductor variable. For zero valuation coefficients, the index of the summation will start at the integer $l$, the index corresponding to the lowest zero valuation coefficient. As $m'_1$ is the smallest exponent appearing in the relevant summation which is coprime to $p$, we have that 
$m'_1=-m_1(p-1)$.

\textbf{(a2)}  In this case $\mathrm{gcd}(l,p) \not=1$. Again, we take $T= Z^p \left(  1  + B \right)^{1/m_1}$ where $B$ is as defined previously. We then base change the $\mu_p$-torsor equation of $X_{2,b} \rightarrow X_b$ to $X_{1,b}$ to obtain:
$Z_2^p  = \sum_{i \in \mathbb{Z}} a_i T^i = 1 +  \sum_{ \substack{ v_K(a_i)=0 \\ i \geq m_2 }} a_i T^i + \sum_{v_K(a_i)>0} a_i T^i
= 1 +  \sum_{ \substack{ v_K(a_i)=0 \\ i \geq m_2 }} a_i Z^{ip} \left(  1  + B \right)^{i/m_1}+ \sum_{v_K(a_i)>0} a_i Z^{ip} \left(  1  + B \right)^{i/m_1}.$
For the terms where $v_K(a_i)=0$, we can express $a_i = b_i^p + c_i$ for some $b_i, c_i \in R$ with $v_K(b_i)=0$ and $v_K(c_i)>0$ so that:
$Z_2^p  = 1 +  \sum_{v_K(b_i)=0} b_i^p Z^{ip} \left(  1  + B \right)^{i/m_1}+ \sum_{v_K(c_i)>0} c_i Z^{ip} \left(  1  + B \right)^{i/m_1} 
+ \sum_{v_K(a_i)>0} a_i Z^{ip} \left(  1  + B \right)^{i/m_1} 
= 1 +  \sum_{v_K(b_i)=0} b_i^p Z^{ip} \left(  1  + B \right)^{i/m_1}+ \sum_{v_K(d_i)>0} d_i Z^{ip} \left(  1  + B \right)^{i/m_1},$ 
where we again set $d_i = c_i $ if $v_K(a_i)=0$ and $d_i=a_i$ if $v_K(a_i)>0$. Now we continue by expansion of the binomial terms:
$Z_2^p = 1 +  \sum_{v_K(b_i)=0} b_i^p Z^{ip} \left(  1  + \frac{i}{m_1} B +... \right)+ \sum_{v_K(d_i)>0} d_i Z^{ip} \left(  1  + \frac{i}{m_1} B  + ... \right)
= 1 +  \sum_{v_K(b_i)=0} \left( b_i Z^{i} \right)^p   +  \sum_{v_K(b_i)=0} \frac{i}{m_1}  b_i^p Z^{ip}B  
\newline
+ \sum_{v_K(d_i)>0} d_i Z^{ip} + \sum_{v_K(d_i)>0} \frac{i}{m_1} d_i Z^{ip} B  + \text{higher order terms}. $

The summation $\sum_{v_K(d_i)>0} d_i Z^{ip}$ does not contribute to the conductor variable since the powers of $Z$ involved are $p$-powers so we can safely exclude it
(cf. Proof in case (a1)).
Now using the identity $(**)$, and after multiplying by a suitable $p$-power, we obtain an equation:
$Z_2^p   =1 - p \sum_{v_K(b_i)=0}  b_i Z^i +  \sum_{v_K(b_i)=0} \frac{i}{m_1}  b_i^p Z^{ip}B  + \text{higher order terms} 
=1 - p \sum_{i \geq m_2}  b_i Z^i +  \sum_{i \geq m_2} \frac{i}{m_1}  b_i^p Z^{ip}B  + \text{higher order terms} $, which equals:

\noindent
$1 - p \sum_{i \geq m_2}  b_i Z^i +  
\sum_{i \geq m_2} \frac{i}{m_1}  b_i^p Z^{ip}
\left( \sum_{k=1}^{p-1} \binom{p}{k} \pi^{-n(p-k)} Z^{-m_1(p-k)} \right) 
+ \text{higher order terms} 
=1 - \pi^{v_K(p)} \sum_{i \geq m_2}  b_i Z^i + \pi^{v_K(p)-n(p-1)} \sum_{i \geq m_2} \frac{i  b_i^p}{m_1}   Z^{ip-m_1(p-1)} + \text{higher order terms};$
up to multiplying the coefficients by units. The second summation has the smallest $\pi$ valuation and so the conductor variable is $m'_1=m_2 p-m_1(p-1)$. 

We now determine $m'_2$. We choose the parameter $T$ so that $u_1 = \sum_{i \in \mathbb{Z} } a_i T^i$ and $u_2 = T^h$ in the case (a1) while $u_2 = 1 +T^{m_2}$  in the case (a2).

\textbf{(a1) }In this case, $m_2=0$. The parameter of $X_{2,b}$ is $Z:=Z_2^{1/h}$ where $T=Z^p$ is obtained from the torsor equation $Z_2^p = T^h \Leftrightarrow \left( Z_2^{1/h} \right)^p =T$. We base change the torsor equation of $X_{1,b} \rightarrow X_b$ to $X_{2,b}$ to obtain:
$(1+\pi^n Z_1)^p  =1+ \pi^{np} \sum_{i \in \mathbb{Z} } a_i T^i =1+ \pi^{np} \sum_{i \in \mathbb{Z} } a_i Z^{pi}   
= 1+ \pi^{np} \sum_{v_K(a_i)=0} a_i Z^{pi}+ \pi^{np} \sum_{v_K(a_i)>0} a_i Z^{pi}. $
For the terms where $v_K(a_i)=0$, we can express $a_i = b_i^p + c_i$ for some $b_i, c_i \in R$ with $v_K(b_i)=0$ and $v_K(c_i)>0$. Hence
$(1+\pi^n Z_1)^p = 1+ \pi^{np} \sum_{v_K(b_i)=0} b^p_i Z^{pi}+\pi^{np} \sum_{v_K(c_i)>0} c_i Z^{pi}+ \pi^{np} \sum_{v_K(a_i)>0} a_i Z^{pi}   
= 1+  \sum_{v_K(b_i)=0} \left( \pi^{n} b_i Z^{i} \right)^p+ \sum_{v_K(c_i)>0} \pi^{np} c_i Z^{pi}+  \sum_{v_K(a_i)>0} \pi^{np} a_i Z^{pi}   
= 1+  \sum_{v_K(b_i)=0} \left( \pi^{n} b_i Z^{i} \right)^p+\sum_{v_K(d_i)>0} \pi^{np} d_i Z^{pi},$  
where we again set $d_i = c_i$ if $v_K(a_i)=0$ and $d_i =a_i$ if $v_K(a_i)>0$. 
Using the identity $(**)$, and after multiplying by a suitable $p$-power, we get:
$(1+\pi^n Z_1)^p = 1 - p \sum_{v_K(b_i)=0}  \pi^{n} b_i Z^{i} + \sum_{v_K(d_i)>0} \pi^{np} d_i Z^{pi} + \text{higher order terms}
= 1 - \pi^{v_K(p)+n} \sum_{i \geq  m_1}  b_i Z^{i} + \text{higher order terms}$;
up to multiplying the coefficients by units. Now, $m'_2$ is the smallest exponent appearing in this leading summation which is coprime to $p$ and so $m'_2=m_1$.

\textbf{(a2)}  In this case $\mathrm{gcd}(l,p) \not=1$. From the torsor equation $Z_2^p=1+T^{m_2}$ we can deduce:
$Z_2^p = 1+ T^{m_2}  \Leftrightarrow Z_2^p -1= T^{m_2} 
\Leftrightarrow(Z_2-1)^p  -  \sum^{p-1}_{k=1} \binom{p}{k} (-1)^k Z_2^k   =T^{m_2} 
\Leftrightarrow (Z_2-1)^p \left( 1  - (Z_2-1)^{-p} \sum^{p-1}_{k=1} \binom{p}{k} (-1)^k Z_2^k \right)  =T^{m_2} $ which implies:
\newline
$T=\left( (Z_2 - 1 )^{\frac{1}{m_2}} \right)^p \left( 1 - (Z_2 - 1 )^{-p} \sum^{p-1}_{k=1} \binom{p}{k} (-1)^k Z_2^k \right)^{\frac{1}{m_2}}$,
and so we take $Z := (Z_2 -1)^{1/m_2}$, which we already know to be the parameter of $X_{2,b}$ by case 2 of this proof, in order to write:
$T =  Z^p \left( 1 - Z^{-m_2 p} \sum^{p-1}_{k=1} \binom{p}{k} (-1)^k (1+Z^{m_2})^k \right)^{\frac{1}{m_2}}. $
For simplicity, let us denote $ \sum^{p-1}_{k=1} \binom{p}{k} (-1)^k (1+Z^{m_2})^k$ by $B$ so that we can write $T=Z^p \left( 1 - Z^{-m_2 p} B \right)^{\frac{1}{m_2}}$. We can now base change the torsor equation of $X_{1,b} \rightarrow X_b$ to $X_{2,b}$ to obtain:
$(1+\pi^n Z_1)^p  =1+ \pi^{np} \sum_{i \in \mathbb{Z} } a_i T^i 
=1+ \pi^{np} \sum_{i \in \mathbb{Z} } a_i Z^{pi} \left( 1 - Z^{-m_2 p}B\right)^{\frac{i}{m_2}}  
=1+ \pi^{np} \sum_{i \in \mathbb{Z} } a_i Z^{pi} \left( 1 - \frac{i}{m_2}Z^{-m_2 p}B + ... \right) 
=1+ \pi^{np} \sum_{i \in \mathbb{Z} } a_i Z^{pi} 
\newline
-  \pi^{np} \sum_{i \in \mathbb{Z} } \frac{i a_i }{m_2}Z^{p(i-m_2)}B+ ... $. 
Partitioning any summation above over the index $i$ into the terms where $v_K(a_i)=0$ and the terms where $v_K(a_i)>0$ gives:
$(1+\pi^n Z_1)^p =1+ \pi^{np} \sum_{v_K(a_i)=0 } a_i Z^{pi} -  \pi^{np} \sum_{v_K(a_i)=0} \frac{i a_i }{m_2}Z^{p(i-m_2)}B 
+ \pi^{np} \sum_{v_K(a_i)>0 } a_i Z^{pi} -  \pi^{np} \sum_{v_K(a_i)>0} \frac{i a_i }{m_2}Z^{p(i-m_2)}B + ... $. 
Again, for the terms with $v_K(a_i)=0$, we can express $a_i = b_i^p + c_i$ for some $b_i, c_i \in R$ with $v_K(b_i)=0$ but $v_K(c_i)>0$ and so we have: 

$(1+\pi^n Z_1)^p  =1+ \pi^{np} \sum_{v_K(b_i)=0 } b_i^p Z^{pi} + \pi^{np} \sum_{v_K(c_i)>0 } c_i Z^{pi} 
-  \pi^{np} \sum_{v_K(b_i)=0} \frac{i b_i ^p}{m_2}Z^{p(i-m_2)} B - \pi^{np} \sum_{v_K(c_i)>0} \frac{i c_i }{m_2}Z^{p(i-m_2)} B 
+ \pi^{np} \sum_{v_K(a_i)>0 } a_i Z^{pi}
- \pi^{np} \sum_{v_K(a_i)>0} \frac{i a_i }{m_2}Z^{p(i-m_2)} B$
 \newline
$ + \text{higher order terms}  
= 1+  \sum_{v_K(b_i)=0 } \left( \pi^{n} b_i Z^{i} \right)^p 
-  \pi^{np} \sum_{v_K(b_i)=0} \frac{i b_i ^p}{m_2}Z^{p(i-m_2)} B 
+ \text{higher order terms}$;
excluding summations whose coefficients have positive valuation. 
Using the identity $(**)$, and after multiplying by a suitable $p$-power, yields:
\newline
$(1+\pi^n Z_1)^p =1 - p \sum_{v_K(b_i)=0}  \pi^n b_i Z^{i} -  \pi^{np} \sum_{v_K(b_i)=0} \frac{i b_i ^p}{m_2}Z^{p(i-m_2)} B 
+ \text{higher order terms}=1 - \pi^{v_K(p)+n} \sum_{i \geq l}  b_i Z^{i} 
 -  \pi^{np} \sum_{v_K(b_i)=0} \frac{i b_i }{m_2}Z^{p(i-m_2)} \left( \sum^{p-1}_{k=1} \binom{p}{k} (-1)^k (1+Z^{m_2})^k \right) 
+ \text{higher order terms}
\newline
=1 - \pi^{v_K(p)+n} \sum_{i \geq l}  b_i Z^{i} +  \pi^{v_K(p)+np} \sum_{v_K(b_i)=0} \frac{i b_i }{m_2}Z^{p(i-m_2)} (1+Z^{m_2}) 
+ \text{higher order terms};$
up to multiplying the coefficients by units. Since $v_K(p)+n$ is the smallest exponent of $\pi$ and $l = \mathrm{min} \{ i | v_K(a_i) = 0, \mathrm{gcd}(i,p) =1 \} = m_1$ is the starting index, we have that $m'_2=m_1$ is the conductor variable.

\begin{center}
\line(1,0){250}
\end{center}

\textbf{5.}  $(\mu_p, \mu_p)$.
Here $m_1, m_2 \geq 0$. The first $\mu_p$-torsor is given by the equation $Z_1^p = u_1$  and the second $\mu_p$-torsor by $Z_2^p = u_2$ where $u_1, u_2 \in A^{\times}$.  Modulo $\pi$, these torsor equations reduce to $z_1^p = \bar{u}_1$ and $z_2^p= \bar{u}_2$ with acting group schemes $\mu_p$ on the special fibre. On the special fibre the reduced power series are of the form $\bar{u}=\sum_{i \geq l} \bar{a}_i t^i$ for some integer $l$. Depending on whether or not these $l$ are coprime to $p$ or not, there are three cases to consider. In particular, the pairs $(a1,a1)$, 
$(a1,a2)$ and $(a2,a2)$. We only treat the cases $(a1,a2)$ and $(a2,a2)$, the case $(a2,a1)$ is treated in a similar way to the case $(a1,a2)$. For the case $(a1,a1)$ see Remark 1.8.

\textbf{(a1,a2)} Here $m_1 = 0$ and $m_2 > 0$. We start by computing $m'_1$. We can choose the parameter $T$ so that $u_1 = T^{h}$ but $u_2 = 1 + \sum_{\substack{  v_K(a_i) = 0  \\   i \geq m_2  } } a_i T^i +  \sum_{v_K(a_i) > 0} a_i T^i$ remains as a power series. From the torsor equation $Z_1^p=T^h$, we can write $T =Z^p$ where $Z= Z_1^{1/h}$ is the parameter of $X_{1,b}$. Then we can base change the torsor equation of $X_{2,b} \rightarrow X_b$ to $X_{1,b}$ and obtain:
$Z_2^p =  1 + \sum_{\substack{  v_K(a_i) = 0  \\   i \geq m_2  } } a_i T^i +  \sum_{v_K(a_i) > 0} a_i T^i 
= 1 + \sum_{\substack{  v_K(a_i) = 0  \\   i \geq m_2  } }a_i Z^{pi} +  \sum_{v_K(a_i) > 0} a_i Z^{pi}.$  
For the terms with $v_K(a_i)=0$, we can express $a_i = b_i^p + c_i$ for some $b_i, c_i \in R$ with $v_K(b_i)=0$ and $v_K(c_i)>0$ and so we have:
$Z_2^p = 1 +  \sum_{\substack{  v_K(b_i) = 0  \\   i \geq m_2  } } b_i^p Z^{pi} +\sum_{\substack{  v_K(c_i) > 0  \\   i \geq m_2  } }c_i Z^{pi} +  \sum_{v_K(a_i) > 0} a_i Z^{pi}  
= 1 +\sum_{\substack{  v_K(b_i) = 0  \\   i \geq m_2  } } \left( b_i Z^{i} \right)^p + \sum_{\substack{  v_K(c_i) > 0  \\   i \geq m_2  } } c_i Z^{pi} +  \sum_{v_K(a_i) > 0} a_i Z^{pi}. $
Using the identity $(**)$, and after multiplying by a suitable $p$-power, we get:
$Z_2^p = 1 - p  \sum_{\substack{  v_K(b_i) = 0  \\   i \geq m_2  } }  b_i Z^{i}+ \sum_{\substack{  v_K(c_i) > 0  \\   i \geq m_2  } } c_i Z^{pi} +  \sum_{v_K(a_i) > 0} a_i Z^{pi} +  \text{higher order terms} 
= 1 - \pi^{v_K(p)}  \sum_{\substack{  v_K(b_i) = 0  \\   i \geq m_2  } }  b_i Z^{i}+ \sum_{\substack{  v_K(c_i) > 0  \\   i \geq m_2  } } c_i Z^{pi} +  \sum_{v_K(a_i) > 0} a_i Z^{pi}$ 
\newline
$+  \text{higher order terms}, $
up to multiplying the coefficients by units.
Therefore, $m'_1 = m_2$. Here we ignored the last two summands in the above equality (cf. proof of case 4 computing $m_1'$).

Now we want to determine $m'_2$. We can choose the parameter $T$ so that $u_1 =\sum_{i \in \mathbb{Z}} a_i T^i$ is a power series and $u_2 =1 + T^{m_2}$ is in simplified form. We know from case 4 that the torsor equation $Z_2^p=1 + T^{m_2}$ gives rise to:  
$T =  Z^p \left( 1 - Z^{-m_2 p} \sum^{p-1}_{k=1} \binom{p}{k} (-1)^k (1+Z^{m_2})^k \right)^{\frac{1}{m_2}} $
where $Z := (Z_2 -1)^{1/m_2}$ is the parameter of $X_{2,b}$. We have that
$Z_1^p =  \sum_{i \in \mathbb{Z}} a_i T^i =  \sum_{\substack{ v_K(a_i)=0 \\   i \geq l  } } a_i T^i  +  \sum_{v_K(a_i)>0} a_i T^i $
where $l$ is such that $\mathrm{gcd}(l,p)=1$. We can write $T= Z^p  \left( 1 - Z^{-m_2 p} B \right)^{\frac{1}{m_2}}$ by letting $B=\sum^{p-1}_{k=1} \binom{p}{k} (-1)^k (1+Z^{m_2})^k$. Then we  
base change the torsor equation of $X_{1,b} \rightarrow X_b$ to $X_{2,b}$ and obtain:
$Z_1^p = \sum_{\substack{ v_K(a_i)=0 \\   i \geq l  } } a_i  Z^{ip}  \left( 1 - Z^{-m_2 p} B \right)^{\frac{i}{m_2}} +  \sum_{v_K(a_i)>0} a_i  Z^{ip}  \left( 1 - Z^{-m_2 p} B \right)^{\frac{i}{m_2}}.$
For the terms where $v_K(a_i)=0$, we can express $a_i = b_i^p + c_i$ for some $b_i, c_i \in R$ with $v_K(b_i)=0$ and $v_K(c_i)>0$ and so we have:

\noindent
$Z_1^p =  \sum_{\substack{ v_K(b_i)=0 \\   i \geq l  } }b^p_i  Z^{ip}  \left( 1 - Z^{-m_2 p} B \right)^{\frac{i}{m_2}} +  \sum_{\substack{ v_K(c_i)>0 \\ i \geq l}} c_i  Z^{ip}  \left( 1 - Z^{-m_2 p} B \right)^{\frac{i}{m_2}}$ 
\newline
$+  \sum_{v_K(a_i)>0} a_i  Z^{ip}  \left( 1 - Z^{-m_2 p} B \right)^{\frac{i}{m_2}}=
\sum_{\substack{ v_K(b_i)=0 \\   i \geq l  } } b^p_i  Z^{ip}  \left( 1 - Z^{-m_2 p} B \right)^{\frac{i}{m_2}} $
\newline
$+  \sum_{v_K(d_i)>0} d_i  Z^{ip}  \left( 1 - Z^{-m_2 p} B \right)^{\frac{i}{m_2}},$
where $d_i = c_i $ if $v_K(a_i)=0$ and $d_i=a_i$ for $v_K(a_i)>0$. 
Hence:
$Z_1^p =   \sum_{\substack{ v_K(b_i)=0 \\   i \geq l  } } b^p_i  Z^{ip}  \left( 1 - \frac{i}{m_2} Z^{-m_2 p} B + ... \right) 
+  \sum_{v_K(d_i)>0} d_i  Z^{ip}  \left( 1 - \frac{i}{m_2} Z^{-m_2 p} B + ... \right) 
=  \sum_{\substack{ v_K(b_i)=0 \\ i \geq l} } b^p_i  Z^{ip}  -  \sum_{\substack{ v_K(b_i)=0 \\ i \geq l} }b^p_i  Z^{ip} \frac{i}{m_2} Z^{-m_2 p} B 
+  \sum_{v_K(d_i)>0} d_i  Z^{ip} - \sum_{v_K(d_i)>0} d_i  Z^{ip} \frac{i}{m_2} Z^{-m_2 p} B +  \text{higher order terms}. $

Taking into factor the $p$-power $b_l^p Z^{lp} = \left( b_l Z^l \right)^p$ we obtain:
$Z_1^p  =  1+ \sum_{\substack{ v_K(b_i)=0 \\   i \geq l  } } b_l^{-p} b^{p}_i  Z^{p(i-l)}  -  \sum_{\substack{ v_K(b_i)=0 \\ i \geq l}} b_l^{-p}  b^p_i  Z^{p(i-l)} \frac{i}{m_2} Z^{-m_2 p} B 
+  \sum_{v_K(d_i)>0 } b_l^{-p} d_i  Z^{p(i-l)} - \sum_{v_K(d_i)>0 } b_l^{-p} d_i  Z^{p(i-l)} \frac{i}{m_2} Z^{-m_2 p} B 
+  \text{higher order terms}  
=  1+  \sum_{\substack{ v_K(b_i)=0 \\   i \geq l  } }\left( b_l^{-1} b_i  Z^{i-l} \right)^p -  \sum_{\substack{ v_K(b_i)=0 \\   i \geq l  } } b_l^{-p}  b^p_i  Z^{p(i-l)} \frac{i}{m_2} Z^{-m_2 p} B$ 
\newline
$+  \sum_{v_K(d_i)>0} b_l^{-p} d_i  Z^{p(i-l)} - \sum_{v_K(d_i)>0} b_l^{-p} d_i  Z^{p(i-l)} \frac{i}{m_2} Z^{-m_2 p} B 
+  \text{higher order terms}. $
Using the identity $(**)$, and after multiplying by a suitable $p$-power, we get
$
Z_1^p =   1 - p \sum_{\substack{ v_K(b_i)=0 \\ i > l}} b_l^{-1} b_i  Z^{i-l} -  \sum_{\substack{ v_K(b_i)=0 \\ i \geq l}} b_l^{-p}  b^p_i  Z^{p(i-l)} \frac{i}{m_2} Z^{-m_2 p} B 
+  \sum_{ v_K(d_i)>0 } b_l^{-p} d_i  Z^{p(i-l)} - \sum_{v_K(d_i)>0 } b_l^{-p} d_i  Z^{p(i-l)} \frac{i}{m_2} Z^{-m_2 p} B
+  \text{higher order terms}  
=   1 - p \sum_{\substack{ v_K(b_i)=0 \\ i > l}} b_l^{-1} b_i  Z^{i-l}
-  \sum_{\substack{ v_K(b_i)=0 \\ i \geq l}} b_l^{-p}  b^p_i  Z^{p(i-l)} \frac{i}{m_2} Z^{-m_2 p} \left( \sum^{p-1}_{k=1} \binom{p}{k} (-1)^k (1+Z^{m_2})^k \right)$ 
\newline
$+  \text{higher order terms}  
= 1 - \pi^{v_k(p)} \sum_{\substack{ v_K(b_i)=0 \\ i > l}} b_l^{-1} b_i  Z^{i-l} 
+  \pi^{v_k(p)} \sum_{\substack{ v_K(b_i)=0 \\ i \geq l}} b_l^{-p}  b^p_i \frac{i}{m_2} Z^{p(i-l)-m_2 p} (1+Z^{m_2}) +  \text{higher order terms}$,
up to multiplying the coefficients by units. Then, clearly, $m'_2 = -m_2(p-1)$ is the conductor variable. Here we ignored the term
$\sum_{ v_K(d_i)>0 } b_l^{-p} d_i  Z^{p(i-l)}$ in the above summation (cf. proof of case 4 computing $m_1'$). 

\textbf{(a2,a2)} Here both $m_1, m_2 > 0$ and we start by computing $m'_1$. The parameter $T$ can be chosen so that $u_1 =1+ T^{m_1}$ but $u_2 = \sum_{i \in \mathbb{Z}} a_i T^i$ 
remains as a power series. From the torsor equation $Z_1^p=1 + T^{m_1}$, we know from case 4 that we can write $T$ as: 
$T =  Z^p \left( 1 - Z^{-m_1 p} \sum^{p-1}_{k=1} \binom{p}{k} (-1)^k (1+Z^{m_1})^k \right)^{\frac{1}{m_1}} $
where $Z := (Z_1-1)^{1/m_1}$ is the parameter of $X_{1,b}$. 
As before, for purposes of convenience, we write $T=Z^p \left( 1 - Z^{-m_1 p} B \right)^{\frac{1}{m_1}}$ where $B =  \sum^{p-1}_{k=1} \binom{p}{k} (-1)^k (1+Z^{m_1})^k $. Then we can base change the torsor equation for $X_{2,b} \rightarrow X_b$ to $X_{1,b}$ and obtain:
$Z_2^p =  1 +  \sum_{ \substack{  v_K(a_i) = 0 \\ i \geq m_2}} a_i T^i +  \sum_{v_K(a_i) > 0} a_i T^i 
= 1 + \sum_{ \substack{  v_K(a_i) = 0 \\ i \geq m_2}} a_i Z^{ip} \left( 1 - Z^{-m_1 p} B \right)^{\frac{i}{m_1}} +  \sum_{v_K(a_i) > 0} a_i Z^{ip} \left( 1 - Z^{-m_1 p} B \right)^{\frac{i}{m_1}} 
= 1 +  \sum_{ \substack{  v_K(a_i) = 0 \\ i \geq m_2}} a_i Z^{ip} \left( 1 - \frac{i}{m_1}Z^{-m_1 p} B +...\right)$ 
\newline
$+  \sum_{v_K(a_i) > 0} a_i Z^{ip} \left( 1 -\frac{i}{m_1} Z^{-m_1 p} B+... \right) 
= 1 +  \sum_{ \substack{  v_K(a_i) = 0 \\ i \geq m_2}} a_i Z^{ip} - \sum_{ \substack{  v_K(a_i) = 0 \\ i \geq m_2}} a_i Z^{ip} \frac{i}{m_1}Z^{-m_1 p} B +  \sum_{v_K(a_i) > 0} a_i Z^{ip}  -  \sum_{v_K(a_i) > 0} a_i Z^{ip} \frac{i}{m_1} Z^{-m_1 p} B+ \text{higher order terms}. $
For the terms where $v_K(a_i)=0$, we can express $a_i = b_i^p + c_i$ for some $b_i, c_i \in R$ with $v_K(b_i)=0$ and $v_K(c_i)>0$ to obtain:
$Z_2^p = 1 + \sum_{ \substack{  v_K(b_i) = 0 \\ i \geq m_2 }} b_i^p  Z^{ip} - \sum_{ \substack{   v_K(b_i) = 0 \\ i \geq m_2}} b_i^p Z^{ip} \frac{i}{m_1}Z^{-m_1 p} B
+ \sum_{v_K(c_i) > 0} c_i Z^{ip} 
- \sum_{v_K(c_i) > 0} c_i Z^{ip} \frac{i}{m_1}Z^{-m_1 p} B  
+  \sum_{v_K(a_i) > 0} a_i Z^{ip}  -  \sum_{v_K(a_i) > 0} a_i Z^{ip} \frac{i}{m_1} Z^{-m_1 p} B+ \text{higher order terms} 
= 1 + \sum_{ \substack{   v_K(b_i) = 0 \\  i \geq m_2} } \left( b_i  Z^{i} \right)^p - \sum_{\substack{   v_K(b_i) = 0 \\ i \geq m_2} } b_i^p Z^{ip} \frac{i}{m_1}Z^{-m_1 p} B + \sum_{v_K(d_i) > 0} d_i Z^{ip} - \sum_{v_K(d_i) > 0} d_i Z^{ip} \frac{i}{m_1}Z^{-m_1 p} B + \text{higher order terms}$,
where we take $d_i=c_i$ if $v_K(a_i)=0$ and $d_i=a_i$ with $v_K(a_i)>0$. 
Using the identity $(**)$, and after multiplying by a suitable $p$-power, we get:
$Z_2^p = 1 - p \sum_{\substack{  v_K(b_i) = 0 \\ i \geq m_2} } b_i  Z^{i} 
- \sum_{\substack{  v_K(b_i) = 0 \\ i \geq m_2} } b_i^p Z^{ip} \frac{i}{m_1}Z^{-m_1 p} B + \text{higher order terms}  = 1 - p \sum_{ \substack{  v_K(b_i) = 0 \\ i \geq m_2} } b_i  Z^{i} 
- \sum_{ \substack{  v_K(b_i) = 0 \\ i \geq m_2}} b_i^p Z^{ip} \frac{i}{m_1}Z^{-m_1 p} \left( \sum^{p-1}_{k=1} \binom{p}{k} (-1)^k (1+Z^{m_1})^k \right)
+ \text{higher order terms}.$
Then up to multiplying the coefficients by units we have:

$Z_2^p = 1 - \pi^{v_K(p)} \sum_{ \substack{  v_K(b_i) = 0 \\ i \geq m_2} } b_i  Z^{i} + \pi^{v_K(p)} \sum_{ \substack{  v_K(b_i) = 0 \\ i \geq m_2}} b_i^p  \frac{i}{m_1}Z^{(i-m_1)p} (1+Z^{m_1}) 
+ \text{higher order terms}. $

In order to determine which is the smallest power of $Z$ and hence the conductor variable $m'_1$, we need to compare $m_2$ with $m_2p-m_1(p-1)$ since both summations have coefficients with the same $\pi$ valuation. Note that $m_2 \leq m_2p - m_1(p-1)$ is equivalent to $m_1 \leq m_2$. Assuming $m_1 \leq m_2$, we have $m_1' = m_2$ and, by symmetry, $m_2' = m_1p -m_2(p-1)$.
The case $m_1\leq m_2$ is entirely similar.

\begin{center}
\line(1,0){250}
\end{center}

\textbf{6.} $(\mathcal{H}_{n_1}, \mathcal{H}_{n_2})$.
Here $m_1, m_2 \in \mathbb{Z}$  and both are coprime to $p$.  The $\mathcal{H}_{n_1}$-torsor equation is given by $(\pi^{n_1} Z_1 + 1)^p = 1+ \pi^{pn_1}u_1$ and the 
$\mathcal{H}_{n_2}$-torsor equation by $(\pi^{n_2} Z_2 + 1)^p = 1+ \pi^{pn_2}u_2 $ where $u_1, u_2 \in A^{\times}$.  The two torsors have associated conductor variables $m_1, m_2$ respectively.

We begin by computing $m'_1$. We can choose the parameter $T$ so that $u_1 = T^{m_1}$ but $u_2 = \sum_{i \in \mathbb{Z}} a_i T^i$ remains as a power series. By case 4, we can express $T$ in terms of $Z:=Z_1^{1/m_1}$, the parameter for $X_{1,b}$, as:
$T=Z^p \left( 1+  \sum_{k=1}^{p-1} \binom{p}{k} \pi^{-n_1(p-k)} Z^{-m_1(p-k)} \right)^{\frac{1}{m_1}} $
and, for convenience, we set $B=\sum_{k=1}^{p-1} \binom{p}{k} \pi^{-n_1(p-k)} Z^{-m_1(p-k)}$ so that we can write $T=Z^p \left( 1+  B \right)^{\frac{1}{m_1}}$. Then we base change the torsor equation for $X_{2,b} \rightarrow X_b$ to $X_{1,b}$ to obtain:
$Z_2^p  = 1+  \pi^{pn_2} \sum_{i \in \mathbb{Z}} a_i T^i 
= 1+  \pi^{pn_2} \sum_{ \substack{  v_K(a_i) = 0 \\ i \geq m_2}} a_i T^i + \pi^{pn_2} \sum_{v_K(a_i) > 0} a_i T^i 
= 1+  \pi^{pn_2} \sum_{ \substack{  v_K(a_i) = 0 \\ i \geq m_2}}  a_i Z^{ip} \left( 1+  B \right)^{\frac{i}{m_1}}$
\newline
$+ \pi^{pn_2} \sum_{v_K(a_i) > 0} a_i Z^{ip} \left( 1+  B \right)^{\frac{i}{m_1}}
= 1+  \pi^{pn_2} \sum_{ \substack{  v_K(a_i) = 0 \\ i \geq m_2}}  a_i Z^{ip} \left( 1+ \frac{i}{m_1} B + ... \right) $
\newline
$+ \pi^{pn_2} \sum_{v_K(a_i) > 0} a_i Z^{ip} \left( 1+ \frac{i}{m_1}  B + ... \right)
= 1+  \pi^{pn_2}  \sum_{ \substack{  v_K(a_i) = 0 \\ i \geq m_2}}  a_i Z^{ip}+  \pi^{pn_2} \sum_{ \substack{  v_K(a_i) = 0 \\ i \geq m_2}}  a_i Z^{ip} \frac{i}{m_1} B 
+  \pi^{pn_2} \sum_{v_K(a_i) > 0} a_i Z^{ip} +  \pi^{pn_2} \sum_{v_K(a_i) > 0} a_i Z^{ip}  \frac{i}{m_1}  B
+\text{higher order terms}. $
For the terms where $v_K(a_i)=0$, we can express $a_i = b_i^p + c_i$ for some $b_i, c_i \in R$ with $v_K(b_i)=0$ and $v_K(c_i)>0$ to obtain:
$Z_2 ^p = 1+  \pi^{pn_2} \sum_{ \substack{ v_K(b_i) = 0 \\ i \geq m_2}} b_i^p Z^{ip}+  \pi^{pn_2} \sum_{ \substack{ v_K(b_i) = 0 \\ i \geq m_2}} b^p_i Z^{ip} \frac{i}{m_1} B 
+  \pi^{pn_2} \sum_{v_K(c_i) > 0} c_i Z^{ip}+  \pi^{pn_2} \sum_{v_K(c_i) > 0} c_i Z^{ip} \frac{i}{m_1} B 
+  \pi^{pn_2} \sum_{v_K(a_i) > 0} a_i Z^{ip} +  \pi^{pn_2} \sum_{v_K(a_i) > 0} a_i Z^{ip}  \frac{i}{m_1}  B + \text{higher order terms} 
= 1+   \sum_{ \substack{ v_K(b_i) = 0 \\ i \geq m_2}} 
\left( \pi^{n_2} b_i Z^{i} \right)^p+  \pi^{pn_2}  \sum_{ \substack{ v_K(b_i) = 0 \\ i \geq m_2}} b^p_i Z^{ip} \frac{i}{m_1} B 
+  \pi^{pn_2} \sum_{v_K(d_i) > 0} d_i Z^{ip}+  \pi^{pn_2} \sum_{v_K(d_i) > 0} d_i Z^{ip} \frac{i}{m_1} B 
+ \text{higher order terms},$
where $d_i = c_i$ if $v_K(a_i)=0$ and $d_i=a_i$ if $v_K(a_i)>0$. 
Using the identity $(**)$, and after multiplying by a suitable $p$-power, we get:

$(\pi^{n_2} Z_2 + 1)^p = 1 - p  \sum_{ \substack{ v_K(b_i) = 0 \\ i \geq m_2}}  \pi^{n_2} b_i Z^{i}  
+  \pi^{pn_2} \sum_{ \substack{ v_K(b_i) = 0 \\ i \geq m_2}} b^p_i Z^{ip} \frac{i}{m_1} B  +  \text{higher order terms} 
= 1 - p \sum_{ \substack{ v_K(b_i) = 0 \\ i \geq m_2}}  \pi^{n_2} b_i Z^{i}  
+  \pi^{pn_2}  \sum_{ \substack{ v_K(b_i) = 0 \\ i \geq m_2}} b^p_i Z^{ip} \frac{i}{m_1} \left( \sum_{k=1}^{p-1} \binom{p}{k} \pi^{n_1(k-p)} Z^{-m_1(p-k)} \right)  
+  \text{higher order terms}. $
Then up to multiplying the coefficients by units we have:
$Z_2 ^p = 1 - \pi^{v_K(p)+n_2}   \sum_{ \substack{ v_K(b_i) = 0 \\ i \geq m_2}} b_i Z^{i} + \pi^{v_K(p)+pn_2-n_1(p-1)} \sum_{ \substack{ v_K(b_i) = 0 \\ i \geq m_2}} b^p_i Z^{ip} \frac{i}{m_1}  Z^{-m_1(p-1)} 
+  \text{higher order terms}.$

We have to compare $v_K(p)+n_2$ with $v_K(p)+pn_2-n_1(p-1)$ in order to determine the smallest power of $\pi$. Note that $v_K(p)+n_2 < v_K(p)+pn_2-n_1(p-1)$ is equivalent to $n_1 < n_2$ and so when this happens, $m'_1 = m_2$ and when $n_1 > n_2$, we have $m'_1 = m_2p-m_1(p-1)$. We also need to consider the case where $n_1 = n_2$. Comparing $m_2$ with $m_2p-m_1(p-1)$ we have that:
\[m'_1= \mathrm{min} \{ m_2  , m_2p-m_1(p-1) \} = \begin{cases}  m_2 & \text{if $m_1 < m_2$} \\   m_2p-m_1(p-1) & \text{if $m_1 \geq m_2$} \end{cases}\]
Now we want to determine $m'_2$ but by the symmetry present in this case this is entirely similar to the above consideration. In particular, if $n_1 < n_2$ then $m'_2 = m_1p-m_2(p-1) $, if $n_1 > n_2$ then $m'_2 =m_1$ and, finally, if $n_1 = n_2$ then:
\[m'_2= \mathrm{min} \{ m_2  , m_1p-m_2(p-1) \} = \begin{cases} m_1p-m_2(p-1) & \text{if $m_1 < m_2$} \\   m_1 & \text{if $m_1 \geq m_2$} \end{cases}\]

\begin{center}
\line(1,0){250}
\end{center}

All six possible cases have now been treated.
\end{proof}

We are also able to state when the Galois cover $Y_b\to X_b$ has a torsor structure by taking into account when base changing in the above proof without additional modification resulted in the equation of the normalisation (see also Theorem 3.4).

\begin{thmB} Let $f_i :  X_{i,b} \rightarrow X_b$ be non-trivial Galois covers of degree $p$ above the formal boundary $X_b$ which are generically disjoint for $i=1,2$. Let $G_i$ be the 
corresponding group schemes for $i=1,2$ and let $Y_b$ be as defined in Theorem 1.1. Then $Y_b = X_{1,b} \times_{X_b} X_{2,b}$, in which case  $Y_b \rightarrow X_b$ is a torsor under $G_1 \times _{\Spec R} G_2$, if and only if at least one of the two group schemes $G_i$ is the \'{e}tale group scheme $\mathcal{H}_{v_K(\lambda)}$.
\end{thmB}

\begin{proof} See Theorem 3.5.
\end{proof}



\begin{defA} For the extension $B/A$ of DVR's where $X_b =\mathrm{Spf} \left( A \right)$ and $Y_b =\mathrm{Spf} \left( B \right)$ are as in Theorem 1.1, we define the \textbf{special different(s)} by
\[d_{s_1} = (c_1 - 1)p(p-1) + (c'_1-1)(p-1)\]
\[d_{s_2} = (c_2 - 1)p(p-1) + (c'_2-1)(p-1)\]
\end{defA}

\begin{lemA}The above two special differents are in fact equal: $d_{s_1} = d_{s_2}$.
\end{lemA}

\begin{proof}Follows immediately by substituting the possible values for $c'_1$ and $c'_2$ under each of the six cases given in Corollary 1.2.
\end{proof}

This $d_{s_i}$, $i=1,2$, coincides in fact with the term $\varphi(s)$ which appears in Kato's vanishing cycles formula in the case of a Galois cover of type $(p,p)$ 
(Theorem 6.7 in [Kato]). We will also see this variable makes an appearance in our genus formula in Theorem 2.2.1 in the next section.

\begin{corB}We have the following relationship between conductors: \[c'_2-c'_1 =(c_1  - c_2 )p.\]
\end{corB}

\begin{proof}Follows from rearranging the relationship between conductors given by $d_{s_1} = d_{s_2}$.
\end{proof}

Had Corollary 1.5 been given, we would have only needed to perform half of the computations in the proof of Theorem 1.1. In particular, with $m'_1$ (or equivalently $c'_1$) obtained, $m'_2$ (or equivalently $c'_2$) would be determined by this relationship. In fact, this formula could have been derived independently using the theory of higher ramification groups as per [Serre] but only for 
the $\left( \mathcal{H}_{v_K(\lambda)}, \mathcal{H}_{v_K(\lambda)} \right)$ case, the first of the six cases in Theorem 1.1. This is because this ramification theory only holds when the residue field extension is separable and so both group schemes $G_i$, $i=1,2$, must be \'etale.

From the calculations in the proof of Theorem 1.1 it is also possible to compute the degree of the different $\delta$ in the extension $B/A$. 
Since the ramification index of this extension $e=1$, we see easily that $\delta = \delta_1 +  \delta_1'=\delta_2 + \delta'_2$ (cf. Theorem 1.6 for notations).
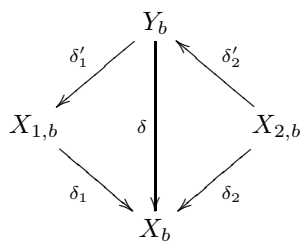
\begin{equation*}
 \xymatrix{
& Y_b \ar@{->}[dl]_{ \delta'_1 } \ar@{->}[dd]_{\delta} \ar@{<-}[dr]^{ \delta'_2} \\
X_{1,b} && X_{2,b}  \\
& X_b  \ar@{<-}[ul]^{\delta_1} \ar@{<-}[ur]_{ \delta_2 }
}
\end{equation*}

\begin{thmE}With the situation described in Theorem 1.1, let $\delta'_i$ (resp. $\delta_i$) denote the degree of the different corresponding to the extension $Y_b \rightarrow X_{i,b}$
(resp. $X_{i,b}\to X_b$), $i=1,2$. Then, for all possible pairs $(G_1,G_2)$, we can explicitly state the values for $\delta'_i$ as follows:

\begin{center}\small
\begin{longtable}{ | p{3cm}| p{5cm} | p{5cm} |}
\caption[Degree of the differents $\delta_1' $, $\delta'_2 $  ]{Degree of the differents $\delta_1' $, $\delta'_2 $ }\\

\hline \multicolumn{1}{|c|}{$\left( G_{1}, G_{2} \right)$} & \multicolumn{1}{c|}{$\delta_1' $} & \multicolumn{1}{c|}{$ \delta'_2 $} \\ \hline \hline
\endfirsthead

\multicolumn{3}{c}%
{{}} \\
\hline \multicolumn{1}{|c|}{$\left( G_{1}, G_{2} \right)$} &
\multicolumn{1}{c|}{$\delta_1' $} &
\multicolumn{1}{c|}{$ \delta'_2 $} \\ \hline 
\endhead

\hline \multicolumn{3}{|r|}{{Continued on next page}} \\ \hline
\endfoot

\hline \hline
\endlastfoot

\begin{center}$\mathcal{H}_{v_K(\lambda)}, \mathcal{H}_{v_K(\lambda)}$\end{center}
& \begin{center}$0$\end{center}
&  \begin{center}$0$\end{center}
\\ \hline

\begin{center}$\mathcal{H}_{v_K(\lambda)},\mu_p$\end{center}
&  \begin{center}$v_K(p)$\end{center}
& \begin{center}$0$\end{center}
\\ \hline

\begin{center}$\mathcal{H}_{v_K(\lambda)},\mathcal{H}_{n}$\end{center}
& \begin{center}$v_K(p)- n(p-1)$\end{center}
&  \begin{center}$0$\end{center}
\\ \hline

\begin{center}$\mathcal{H}_{n},\mu_p$\end{center}
&\begin{center} $v_K(p)- \frac{v_K(p) -n (p-1)}{p} (p-1)$\end{center}
& \begin{center}$v_K(p)- \frac{v_K(p) +n}{p} (p-1)$\end{center}
\\ \hline

\begin{center}$\mu_p,\mu_p$\end{center}
& \begin{center}$v_K(p)- \frac{v_K(p)}{p} (p-1)$\end{center}
& \begin{center}$v_K(p)- \frac{v_K(p)}{p} (p-1)$\end{center}
\\ \hline

\begin{center}$\mathcal{H}_{n_1},\mathcal{H}_{n_2}$\end{center}
& If $n_1 \leq n_2$ then $\delta'_1$ equals: 

\begin{center}$v_K(p)- \frac{v_K(p)+n_2}{p} (p-1)$\end{center}

If $n_1 > n_2$ then $\delta'_2$ equals:  

\begin{center}$v_K(p)- \frac{v_K(p)+n_2p-n_1(p-1)}{p} (p-1)$\end{center}

& If $n_1 \leq n_2$ then $\delta'_2$ equals: 

\begin{center}$v_K(p)- \frac{v_K(p)+n_1p - n_2(p-1)}{p} (p-1)$\end{center}

If $n_1 > n_2$ then $\delta'_2$ equals:

\begin{center}$v_K(p)- \frac{v_K(p)+n_1}{p} (p-1)$\end{center}

\end{longtable}
\end{center}

\end{thmE}

\begin{proof} For an arbitrary rank $p$ torsor with conductor variable $m$ and torsor equation $Z^p=1+\pi^{np}T^m + \text{higher terms}$, 
where $0 \leq n \leq v_K(\lambda)$, the degree of the different is given by $\delta = v_K(p)-n(p-1)$. This means that computing the degree of the different $\delta$ reduces to obtaining $n$ from the exponent of $\pi$ in the coefficient of the term corresponding to the conductor variable $m$. From our calculations obtained in the proof of Theorem 1.1, we can simply read off the $n$ value in each of the cases and substitute into the formula $v_K(p)-n(p-1)$ to obtain the degree of the different at that particular stage. Strictly speaking, we can only rely on this approach for the last three cases because in the first three cases we worked modulo $\pi$ on the special fibre. In the first three cases to compute $\delta$ we use the fact that the degree of the different is preserved by \'etale base change. 
\end{proof}
\noindent
{\bf Example 1.7.} With the type $(p,p)$ case established it is possible to manually perform the same calculations in the type $(p,...,p)$ setting. We illustrate this with a type $(p,p,p)$ example, using the 
$(p,p)$-type results in  Theorem 1.1 and Corollary 1.2 iteratively at each stage to determine the conductors in terms of the base level conductors. 
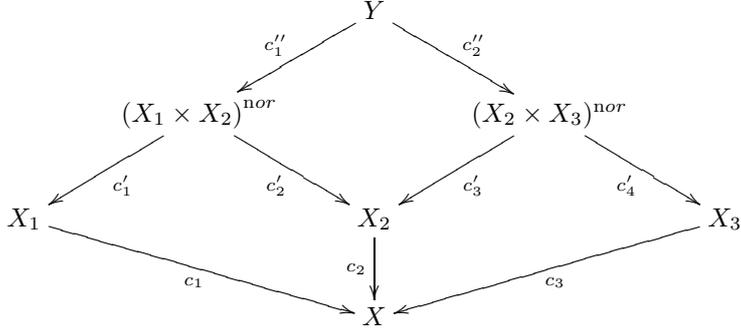
\begin{equation*}
 \xymatrix{
&& Y \ar[dl]_{c''_{1}}  \ar[dr]^{c''_{2}} \\
& {(X_1 \times X_2 )}^{\text nor}\ar[dl]^{c'_{1}} \ar[dr]_{c'_{2}} && { (X_2 \times X_3)^{\text nor}} \ar[dl]^{c'_{3}}  \ar[dr]_{c'_{4}}  \\
X_1 \ar[drr]_{c_1} && X_2 \ar[d]_{c_2} && X_3 \ar[dll]^{c_3} \\
&& X}
\end{equation*}
Suppose $X_i \rightarrow X$ are torsors under the $R$-group scheme $G_i$ of rank $p$ for $i=1,2,3$ which are pairwise generically disjoint, i.e $X_i$ is generically disjoint from $X_j$ for $i\neq j$. 
Write $(X_1 \times X_2)^{\text nor}$ and $(X_2 \times X_3)^{\text nor}$ for the normalisation of $X$ in $X_{1,K}\times_{X_K} X_{2,K}$ and $X_{2,K}\times_{X_K} X_{3,K}$, respectively, and $Y$ for the normalisation of $X$ in $X_{1,K}\times_{X_K} X_{2,K}\times_{X_K} X_{3,K}$.
For the purposes of an example, let $G_i = \mathcal{H}_{v_K(\lambda)}$ for all $i$ and assume $c_1 \leq c_2 \leq c_3$ and $c'_2 \leq c'_3$. Then, by applying the type $(p,p)$ formula iteratively we can compute the conductor $c''_1$ as follows (similarly one can compute $c''_2$):
$c''_1 = c'_3p-c'_2(p-1) = \left( c_3p-c_2(p-1) \right)p-\left( c_1 \right)(p-1) 
= c_3p^2-c_2p(p-1)-c_1(p-1).$

\medskip
\noindent
{\bf Remark 1.8}
We discuss an example which illustrates the case $\bold {(a1,a1)}$ occuring in the proof of Theorem 1.1, the case $(\mu_p,\mu_p)$.  
Here $m_1 = m_2 = 0$. In this case one can show that the group schemes acting on the torsors $Y_b \rightarrow X_{i,b}$ with conductor variables $m'_i$ are $\mathcal{H}_{n'_i}$ with $0 < n'_i < v_K(\lambda)$ for $i=1,2$. Moreover, one can show $n'_1=n'_2$ and $m'_1 = m'_2$. Suppose $u_1 = T^h$ and $u_2=T^l \left( v(T) \right)$ where $v(T)= \sum_{i \geq 0} a_i T^i$ such that $a_0 \in R^{\times}$ is a unit. 
The conductor variables $m'_1, m'_2$ are in fact encoded in $v(T)$. The proof is complicated to present in general. 
Instead, we treat an instructive example to illustrate the computations involved.

Suppose $p \not=2 $, $u_1 = T$ and $u_2=T + T^3=T(1+T^2)$. Then the $\mu_p$-torsors above $X_b$ generically defined by $Z_1^p =T$ and $Z_2^p =T(1+T^2)$ are linearly disjoint.
We begin by computing $m'_1$. We can write $T =Z^p$ where $Z= Z_1$ is the parameter of $X_{1,b}$. Then we can base change the torsor equation for $X_{2,b} \rightarrow X_b$ to $X_{1,b}$:
$Z_2^p =  T(1+T^2) =Z^p(1+Z^{2p}). $
Removing the multiplicative factor $Z^p$ (which is a $p$-power) gives rise to an equation of the form:
$Z_2^p =  1+Z^{2p} = \left( 1+Z^{2} \right)^p - \sum_{k=1}^{p-1} \binom{p}{k} Z^{2k}.$
Multiplying this equation by the $p$-power $\left( 1+Z^{2} \right)^{-p} = 1 -pZ^2 + ... $, results in an equation of the form:
$Z_2^p =  1 - \sum_{k=1}^{p-1} \binom{p}{k} Z^{2k}+ \text{higher order terms}. $
The smallest power of $Z$ which is coprime to $p$ is obtained when $k=1$. Therefore, $m'_1=2$.

Now we want to determine $m'_2$.  We have that $Z_2^p =T(1+T^2)$ which we can write $Z_2^p=T' \Leftrightarrow Z^p=T'$ where the parameter of $X_{2,b}$ is $Z=Z_2$, and the relation $T'=T(1+T^2)$ implies:
$T =  T' \left( 1+T^2 \right)^{-1} =  T' \left( 1 - T^2 + T^4 - T^6 + ...  \right)  
=  T'  + \left( - T' T^2 + T' T^4 - T' T^6 + ...  \right).$
From this we deduce that $T$ can be expressed as $T' + $ higher powers of $T'$.  In particular, $T =  T' - T'^3 + T'^5 - T'^7 + ... $. We can now proceed to base change the torsor equation of $X_{1,b} \rightarrow X_b$ to $X_{2,b}$:
$Z_1^p = T= T' - T'^3 + T'^5 - T'^7 +  \text{higher order terms}  
= Z^p - Z^{3p} + Z^{5p} - Z^{7p}+  \text{higher order terms}  
= Z^p \left( 1 - Z^{2p} + Z^{4p} - Z^{6p}+  \text{higher order terms} \right).$
Removing the multiplicative factor $Z^p$, gives rise to an equation of the form:
$Z_1^p = 1 - Z^{2p} + Z^{4p} - Z^{6p} +  \text{higher order terms} 
= \left(  1 - Z^{2} + Z^{4}  ... \right)^p  - p \sum (Z^2 + Z^4  + ... ) + \text{higher order terms}$;
by using the formula $(**)$. Multiplying this equation by the inverse $p$-power $ \left(  1 - Z^{2} + Z^{4}  ... \right)^{-p}$, results in an equation of the form:
$Z_1^p =1  - p \sum (Z^2 + Z^4  + ... ) + \text{higher order terms} 
=1  - \pi^{v_K(p)} \sum (Z^2 + Z^4  + ... ) + \text{higher order terms}$;  
up to multiplying the coefficients by units. Therefore, the conductor variable is $m'_2=2=m'_1$ and $n'_2=\frac{v_K(p)}{p}=n'_1$.

\section*{\S2 Computation of vanishing cycles}

\section*{\S 2.1 Computation of vanishing cycles in Galois cover of degree $p$}
In this section we recall some results from [Sa\"\i di1].
\begin{defB}For an $R$-curve $X$ and a closed point $x\in X$, we let $\mathcal{X} := \mathrm{Spf} \hat{ \mathcal{O}}_{X,x} $ 
denote the formal spectrum of the completion of the local ring of $X$ at $x$.  Assume $X_k$ is reduced.
Then the \textbf{genus of the point} $x$ is given by:
\[g_x := \delta_x - r_x +1 \]
where
\begin{itemize}
\item $\delta_x = \mathrm{dim}_k \left( \tilde{\mathcal{O}}_x / \mathcal{O}_x \right)$, 
\item $r_x$ is the number of maximal ideals in $\tilde{\mathcal{O}}_x$. 
\end{itemize}
Here $\mathcal{O}_x\defeq \hat {\mathcal O}_{X,x}/\pi$ denotes the stalk of the special fibre $\mathcal{X}_k$ at $x$ and $\tilde{\mathcal{O}}_x$ denotes its normalisation in its total ring of fractions. 
\end{defB}

If $g_x=0$, the point $x$ is either a smooth or an ordinary multiple point (where $\delta_x = r_x -1$). 
Here is a result which provides an explicit formula---a local Riemann-Hurwitz formula---comparing the (above) genus in a Galois cover of degree $p$.

\begin{thmF}(cf. Theorem 3.4 in [Saidi1]) Let $X := \mathrm{Spf} \left( \hat{\mathcal{O}}_x \right)$ be the formal germ of an $R$-curve at a closed point $x$ with $X_k$ reduced (cf. Notations). 
Let $f :Y \rightarrow X$ be a Galois cover of degree $p$ with $Y$ normal and local. Assume that the special fibre $Y_k$ of Y is reduced. 
Let $\{\wp_i \}_{i \in I}$ be the minimal prime ideals of $\hat{\mathcal{O}}_x$ which contain $\pi$, and let $X_{b_i} := \mathrm{Spf} \left( \hat{\mathcal{O}}_{\wp_i} \right)$ be the formal completion of the localisation of X at $\wp_i$. For each $i \in I$ the above cover $f$ induces a torsor $f_{b_i} : Y_{b_i} \rightarrow X_{b_i}$ under a finite and flat $R$-group scheme $G_i$ of rank $p$ above the boundary $X_{b_i}$ with conductor $c_i$, we write $c_i=1$ in case this torsor is trivial.
If $y \in Y$ is the closed point of $Y$, then:
\[ 2g_y - 2 = p(2g_x - 2)+ d_\eta - d_s \]
where $g_y$ (resp. $g_x$) denotes the genus of $y$ (resp. $x$), $d_\eta$ is the degree of the divisor of ramification in the morphism $f_K : Y_K \rightarrow X_K$ 
induced by $f$ on the generic fibre, and $d_s = \sum_{i \in I} (c_i-1)(p-1)$.
\end{thmF}

We will refer to this formula simply as the genus formula.
The following corollary is immediate from Theorem 2.1.2.

\section*{\S 2.2 Computation of vanishing cycles in Galois covers of type $(p, p)$}
In this section we prove that the degree $p$ genus formula in Theorem 2.1.2 can be extended to the case of Galois covers of type $(p,p)$. 
Let $f : Y \rightarrow X$ be a Galois cover of type $(p,p)$ where $X$ is a formal germ of an $R$-curve, $Y$ is local and normal, and $Y_k$ is reduced. We can express 
$f$ as the compositum of two, generically disjoint, degree $p$ Galois covers $Y_i \rightarrow X$ with $Y_i$ normal and local for $i=1,2$ as follows:
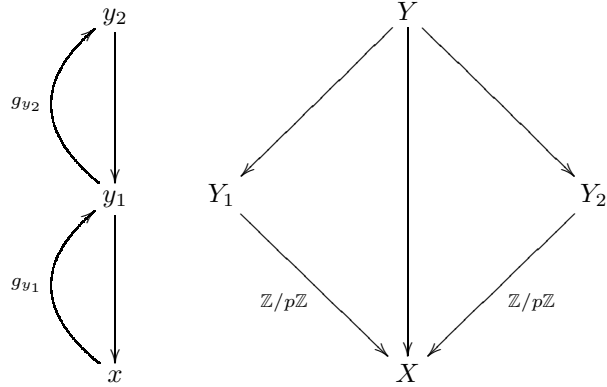
\begin{equation*}
 \xymatrix{
y_2  &&& Y \ar[ddll] \ar[dddd] \ar[ddrr] \\
\\
y_1 \ar@{<-}[uu]  \ar@/^2pc/[uu]^{g_{y_2}} & Y_1 \ar[ddrr]_{\mathbb{Z}/p\mathbb{Z}}  &&&& Y_2 \ar[ddll]^{\mathbb{Z}/p\mathbb{Z}} \\
\\
x  \ar@{<-}[uu]  \ar@/^2pc/[uu]^{g_{y_1}} &&&  X
}
\end{equation*}

Let $\{x_i\}_i \subset X_K$ be the (finite) set of branched points in the cover $f_K:Y_K\to X_K$ between generic fibres and $\{y_{ij}\}_{i,j}\subset Y_K$ the set of ramified points in $f_K$ with $r=\mathrm{Card} \{ \{y_{ij}\}_{i,j} \}$. Thus, for fixed $i$ the $\{y_{ij}\}_{i,j}$ are the points of $Y_K$ above $x_i$.

We assume there are $r_1$ points ($\subset Y_1$) ramified in $Y_{1,K}\to X_K$ and $r_2$ ramified points ($\subset Y$) in $Y_K\to Y_{1,K}$. Because the Galois group $G$ is of type $(p,p)$ an 
inertia subgroup of $G$
has at most cardinality $p$ since $\mathrm {char} (K)=0$ and so the inertia subgroups must be cyclic.
Therefore, two cases occur:  at the first stage we have one point above a branched point then at the second stage there must be $p$ points which sit above it \emph{or} at the first stage we have $p$ points above a branched point then at the second stage there is $1$ point above each of these $p$ points.
In summary, for a branched point $x=x_i$, we have one of the two situations occurring:
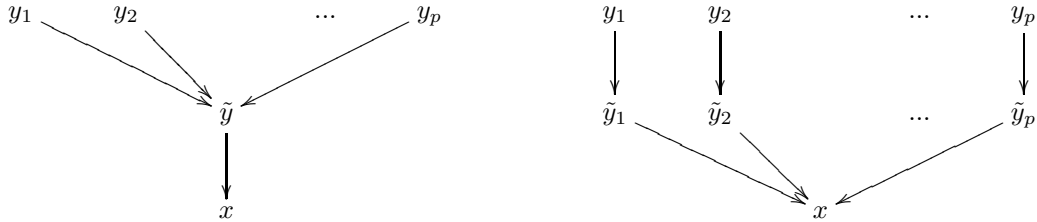
\begin{equation*}
 \xymatrix{
{y}_1 \ar[drr] & {y}_2  \ar[dr] & & ... & {y}_p \ar[dll] && {y}_1 \ar[d] & {y}_2  \ar[d] & & ... & {y}_p \ar[d] \\
&& \tilde y  \ar[d] &&&&  \tilde y_1 \ar[drr] & \tilde y_2  \ar[dr] & & ... & \tilde y_p  \ar[dll] \\
&& x &&& &&& x
}
\end{equation*}

The diagram on the left depicts ramification occurring at the first stage while the diagram on the right depicts ramification occurring at the second stage. Since these two cases are disjoint, this gives us that $r=r_1p+r_2$ where:
\begin{itemize}
\item $r_1 = \mathrm{Card} \{ \text{ ramified points in } Y_1 \rightarrow X \ \}$
\item $r_2 = \mathrm{Card} \{ \text{ ramified points in } Y \rightarrow Y_1 \ \}$
\end{itemize}

For the branched points $\{x_i\}_i \subset X_K$ in the cover $f_K  : Y_K \rightarrow X_K$, we can visualise the general picture, including decomposition groups, as follows:
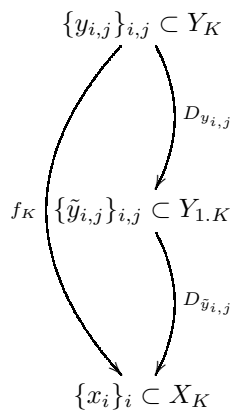
\begin{equation*}
 \xymatrix{
\{ {y}_{i,j}\}_{i,j} \subset Y_K   \ar@/^1pc/[dd]^{D_{{y}_{i,j}}} \ar@/_3pc/[dddd]_{f_K} \\
\\
\{ \tilde y_{i,j}\}_{i,j} \subset Y_{1.K}   \ar@/^1pc/[dd]^{D_{{\tilde y}_{i,j}}} \\
\\
\{ x_{i} \}_i \subset  X_K
}
\end{equation*}
where $D_{\tilde y_{i,j}} \leq \mathbb{Z}/p\mathbb{Z} $ and $D_{{y}_{i,j}}  \leq \mathbb{Z}/p\mathbb{Z} $ denote the decomposition groups of the point $\tilde y_{i,j}$ at the first stage and the point 
${y}_{i,j}$ above $\tilde y_{i,j}$ at the second stage respectively. If only one point sits above $x_i$ in $Y_1$ then the order of the decomposition group $D_{\tilde y_{i,j}}$ 
(resp. $D_{y_{i,j}}$) 
will equal $p$ (resp. $1$) and, otherwise, the opposite is true. This means we have a natural test for ramification in the first and second step as follows:
\[ p = |D_{\tilde y_{i,j}}| \Leftrightarrow p \not= |D_{{y}_{i,j}}| =1 \Leftrightarrow x_i \text{ ramifies at 1st stage} \]
\[ p \not= |D_{\tilde y_{i,j}}| =1 \Leftrightarrow p = |D_{{y}_{i,j}}| \Leftrightarrow x_i \text{ ramifies at 2nd stage} \]

Now we turn to address the decomposition above the boundaries. Let $\{ X_{b_t} \}_t$ denote the boundaries of $X$. 
For each $t$, the Galois cover $f:Y \rightarrow X$ induces a Galois cover $f_t : Y_{b_t} \rightarrow X_{b_t}$ above the boundary $X_{b_t}$
(note that $Y_{b_t}$ is not necessarily connected). 
Unlike the degree $p$ case, the cover $Y_{b_t}\to X_{b_t}$ is not a torsor under $G_{1,t} \times G_{2,t}$ unless at least one of the group schemes $G_{1,t}$ and $G_{2,t}$ is \'{e}tale
(see Theorem 1.2 and Theorem 3.5). However, at each intermediate degree $p$ stage, the corresponding cover is indeed a torsor where $c_{1,t}$ and $c_{2,t}$  (respectively $c'_{1,t}$ and $c'_{2,t}$) are the conductors associated to the torsor under the finite flat $R$-group schemes $G_{1,t}$ and $G_{2,t}$ respectively  (respectively $G'_{1,t}$ and $G'_{2,t}$).
Below is the picture when $Y_{b_t}$ is connected; the case we refer to as being {\it unibranched throughout}.
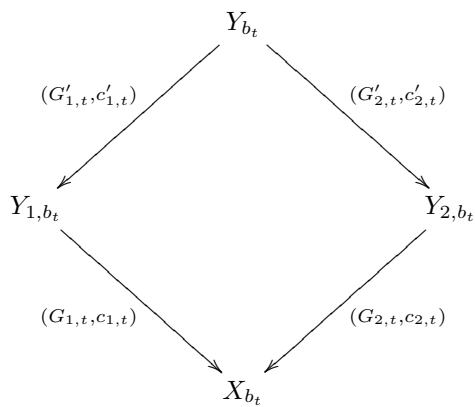
\begin{equation*}
 \xymatrix{
&& Y_{b_t} \ar[ddll]_{(G'_{1,t},c'_{1,t})}  \ar[ddrr]^{(G'_{2,t},c'_{2,t})} \\
\\
Y_{1, b_t} \ar[ddrr]_{(G_{1,t},c_{1,t})}  &&&& Y_{2,b_t} \ar[ddll]^{(G_{2,t},c_{2,t})} \\
\\
&&  X_{b_t}
}
\end{equation*}

Our main Theorem in this section compares the genus in a Galois cover of type $(p,p)$.  

\begin{thmC}Let $X := \mathrm{Spf} \left( \hat{\mathcal{O}}_x \right)$ be the formal germ of an $R$-curve at a closed point $x$ with $X_k$ reduced. Let $f :Y \rightarrow X$ be a Galois cover with group $\mathbb{Z}/p\mathbb{Z} \times \mathbb{Z}/p\mathbb{Z} $---that is, of type $(p,p)$---where $Y$ is normal and local and the special fibre $Y_k$ of Y is reduced.

Let $f_1 : Y_1 \rightarrow X$ and $f_2: Y_2 \rightarrow X$ be two generically disjoint degree $p$ Galois covers such that $Y$ is the compositum of $Y_1$ and $Y_2$. Let $\{X_{b_t}\}_{t\in I}$ denote the boundaries of $X$. The Galois cover $f_1$  induces a torsor $Y_{1, b_{t}} \rightarrow X_{b_t}$ under a finite and flat $R$-group scheme of type $p$ with conductor $c_{1,t}$ for each $t$. Similarly, $c'_{1,t}$ denotes the conductor associated to the torsor $Y_{b_{t}} \rightarrow Y_{1, b_{t}}$. In cases these torsors are trivial we write $c_{1,t}=1$ and $c'_{1,t}=1$, respectively.
We let $r_1$ (resp. $r_2$) denote the number of  ramified points in $Y_{1,K} \rightarrow X_K$ 
(resp. $Y_K \rightarrow Y_{1,K}$). 

If $y \in Y$ is the closed point of $Y$, then:
\[ 2g_y - 2 = p^2(2g_x - 2)+ d_\eta - d_s \]
where $g_y$ (resp. $g_x$) denotes the genus of $y$ (resp. $x$), $d_\eta:=(r_1+r_2)p(p-1)$ is the degree of the divisor of ramification in the morphism $f_K : Y_K \rightarrow X_K$ induced by $f$ on the generic fibre and 
\begin{equation*}\begin{split}
d_s = & \sum_{\stackrel{\text{boundary unibranched}}{\text{throughout}}} \left[(c'_{1,t}-1)(p-1)+(c_{1,t}-1)p(p-1)\right]\\ 
& \ \ \ + \sum_{\stackrel{\text{boundary unibranched}}{\text{then $p$-branched}} } (c_{1,t} -1)p(p-1) + \sum_{\stackrel{\text{boundary $p$-branched}}{\text{then unibranched}} } (c'_{1,t} -1)(p-1).  \\
\end{split}\end{equation*}
\end{thmC}

\begin{proof}

By Theorem 2.1.2 we have the following genus formula for the degree $p$ Galois cover $Y_1 \rightarrow X$ expressing $g_{y_1}$ in terms of $g_x$ where $y_1$ is the point of $Y_1$ above $x$:
$2g_{y_1} -2 = p (2g_{x}-2)+r_1(p-1) - \sum_{t \in I} (c_{1,t}-1)(p-1). $
Each boundary $X_{b_t}$ is either unibranched or $p$-branched in $Y_1$ and so we can break up the $d_s$ summation as follows:
$$2g_{y_1}-2 = p (2g_{x}-2)+r_1(p-1) - \sum_{\stackrel{t \in I}{\text{$X_{b_t}$ unibranched}}} (c_{1,t}-1)(p-1) - \sum_{\stackrel{t \in I}{\text{$X_{b_t}$ $p$-branched}}} (c_{1,t}-1)(p-1). $$
Also by Theorem 2.1.2 we have the following genus formula for the degree $p$ Galois cover $Y \rightarrow Y_1 $ expressing $g_{y}$ in terms of $g_{y_1}$:
$2g_{y} -2 = p \left( 2g_{y_1}-2 \right)+r_2p(p-1) - \sum_{t \in I} (c'_{1,t}-1)(p-1). $
Again, we rewrite the $d_s$ summation into unibranched or $p$-branched cases:
$$2g_y-2  = p (2g_{y_1}-2)+r_2p(p-1)- \sum_{\stackrel{t \in I}{\text{$Y_{1, b_t}$ unibranched}}} (c'_{1,t}-1)(p-1) - \sum_{\stackrel{t \in I}{\text{$Y_{1, b_t}$ $p$-branched}}} (c'_{1,t}-1)(p-1).$$

Tracing a boundary $X_{b_t}$ through the entire type $(p,p)$ Galois cover $f: Y \rightarrow X$, keeping in mind that under the cover $Y_1 \rightarrow X$ the boundary can be $p$-branched or unibranched and, likewise, under the cover $Y \rightarrow Y_1$, we have four possible cases which can arise. In particular, the boundary is unibranched throughout, unibranched and then $p$-branched $p$-branched and then unibranched or finally $p$-branched throughout. Now, substituting, our first genus formula expressing $g_{y_1}$ in terms of $g_x$ into the second genus formula expressing $g_{y}$ in terms of $g_{y_1}$, will give us a genus formula expressing $g_{y}$ in terms of $g_x$, as required:
 
 $2g_{y} -2 = p (2g_{y_1}-2)+r_2p(p-1) -  \sum_{\stackrel{t \in I}{\text{$Y_{1, b_t}$ uni.}}} (c'_{1,t}-1)(p-1) - \sum_{\stackrel{t \in I}{\text{$Y_{1, b_t}$ $p$-b.}}} (c'_{1,t}-1)(p-1) $ 
\newline
$= p \left( p (2g_{x}-2)+r_1(p-1) - \sum_{\stackrel{t \in I}{\text{$X_{b_t}$ uni.}}} (c_{1,t}-1)(p-1) - \sum_{\stackrel{t \in I}{\text{$X_{b_t}$ $p$-b.}}} (c_{1,t}-1)(p-1) \right) 
+r_2p(p-1) -  \sum_{\stackrel{t \in I}{\text{$Y_{1, b_t}$ uni.}}} (c'_{1,t}-1)(p-1) - \sum_{\stackrel{t \in I}{\text{$Y_{1, b_t}$ $p$-b.}}} (c'_{1,t}-1)(p-1)  
= p^2 (2g_{x}-2)+\left(r_1+r_2 \right) p(p-1) - \sum_{\stackrel{t \in I}{\text{uni., uni.}}} \left [(c_{1,t}-1)p(p-1) + (c'_{1,t}-1)(p-1)\right] 
- \sum_{\stackrel{t \in I}{\text{uni., $p$-b.}}} (c_{1,t}-1)p(p-1)  - \sum_{\stackrel{t \in I}{\text{$p$-b., uni.}}} (c'_{1,t}-1)(p-1)  - \sum_{\stackrel{t \in I}{\text{$p$-b., $p$-b.}}} 0. $

So, we obtain a genus formula in the form 
$2g_y - 2 = p^2(2g_x - 2)+ d_\eta - d_s $
where $d_\eta$ and $d_s$ are as expressed in the statement of the Theorem.
\end{proof}

For illustration purposes, we explain this picture on the boundary in the case of a unibranched throughout cover above
an open disc and, so in what follows $X=\mathrm {Spf} R[[T]]$ and $X_{b_t}=\mathrm {Spf} R[[T]]\{\frac {1}{T}\}$.


\textbf{Case: unibranched throughout.}
Write, as above, $c_1$ for the conductor in the first stage and $c'_1$ for the conductor in the
second stage above $X_{b_t}$. 
By Corollary 2.1.3 we have that
$g_{y_1} = \frac{(r_1-c_1-1)(p-1)}{2} $
and by the genus formula given in Theorem 2.1.2 we can write
$2g_{y} -2 = p (2g_{y_1}-2)+r_2p(p-1) - (c_1'-1)(p-1)$
and so substituting the first equation into the second results in
$2g_y - 2 =  p (2g_{y_1}-2)+r_2p(p-1) - (c_1'-1)(p-1)
=  p \left( (r_1-c_1-1)(p-1)-2 \right)+r_2p(p-1) - (c_1'-1)(p-1) 
= \underbrace{-2p^2+2p^2 }_{0}-2p+(r_1+r_2)p(p-1)-(p-1) \left(c_1'-1+p(c_1+1) \right) 
= p^2(0-2)+(r_1+r_2)p(p-1)-(p-1)\left( c_1'-1+p(c_1+1) -2p \right) 
= p^2(0-2)+\underbrace{(r_1+r_2)p(p-1)}_{d_{\eta}}-\underbrace{(p-1)\left( (c_1'-1)+p(c_1-1) \right)}_{d_s}. $

\begin{center}
\line(1,0){250}
\end{center}

The results from the above discussion are summarised in the above table.

\begin{table}[h]
\centering
\caption[Values for $d_{\eta}$ and $d_s$ in the type $(p,p)$ setting above one boundary]{Values for $d_{\eta}$ and $d_s$ in the $(p,p)$ setting above one boundary}
\begin{tabular}{ | c | c | c | c | }
\hline
1st step & 2nd step & $d_{\eta}$ & $d_s$  \\
\hline
 uni & uni & $(r_1+r_2)p(p-1)$ & $ (c_1'-1)(p-1)+(c_1-1)p(p-1)$  \\
 uni & $p$ & $(r_1+r_2)p(p-1)$ & $p(p-1)(c-1)$ \\ 
 $p$ & uni & $(r_1+r_2)p(p-1)$ & $(p-1)(c-1)$ \\
 $p$ & $p$ & $(r_1+r_2)p(p-1)$ & 0 \\ \hline 
\end{tabular}
\end{table}
From this we can deduce that the general form for the genus formula for $(p,p)$-type covers in case $X=\mathrm {Spf} (\hat O_x)$ has a {\it unique} boundary (or equivalently $X_k$ is unibranch)
is given by 
\[ 2g_y - 2 = p^2(2g_x -2)+d_{\eta} - d_s \]
where $d_\eta=(r_1+r_2)p(p-1)$ and where
\[ d_s =\begin{cases} 
      (c'_{1}-1)(p-1)+(c_{1}-1)p(p-1) & \text{boundary unibranched throughout} \\
      (c-1)p(p-1) & \text{boundary unibranched, then $p$-branched} \\
     (c -1)(p-1)  & \text{boundary $p$-branched, then unibranched} \\
     0 & \text{boundary $p$-branched throughout} 
   \end{cases}
\]

We can derive from the above formula some interesting results:

\begin{proA} Let $X=\mathrm{Spf} \left( R[[T]] \right)$ be the formal germ of an $R$-curve at a smooth point $x$ and let $f: Y \rightarrow X$ be a Galois cover with group  $\mathbb{Z}/p\mathbb{Z} \times \mathbb{Z}/p\mathbb{Z}$. Assume $Y$ is normal and local and that the special fibre $Y_k$ of Y is reduced. Let $X_b = \mathrm{Spf}(R[[T]]\{T^{-1}\})$ be the boundary of $X$ and $f_b:Y_b \rightarrow X_b$ the induced Galois cover on the boundaries.  Let $y$ be the unique closed point of $Y_k$ and $d_\eta:=(r_1+r_2)p(p-1)$ be the degree of the divisor of ramification in the morphism $f_K : Y_K \rightarrow X_K$ induced by $f$ on the generic fibre and $c_1$ and $c'_1$ are as in cases 2 and 3 below and where $c$ is the only acting conductor at the relevant unibranched stage. Then:
\begin{enumerate}
\item If $Y_{k}$ is unibranched above $x$ then $g_y= \frac{(p(r_1+r_2-c_1-1)-c'_1-1)(p-1)}{2}$.
\item The morphism $Y_{k}\to X_k$ is unibranched and then $p$-branched above $x$ then 

$g_y= \frac{(p(r_1+r_2-c-1)-2)(p-1)}{2}$.
\item The morphism $Y_{k}\to X_k$ is $p$-branched and then unibranched above $x$ then 

$g_y= \frac{(p(r_1+r_2-2)-c-1)(p-1)}{2}$.
\item The morphism $Y_{k}$ is $p^2$-branched above $x$ then $g_y= \frac{(p(r_1+r_2-2)-2)(p-1)}{2}$.
\end{enumerate}
\end{proA}

\begin{proof}Follows directly from rearranging the type $(p,p)$ vanishing cycles formula with $g_x=0$.
\end{proof}

In this situation, we have the following test for whether $y$ is a smooth point or not.

\begin{corE}With the same assumptions as in Proposition 2.2.2, 
$y$ is a smooth point if and only if we are in the case 1 of loc. cit. and $p(r_1+r_2-1)=1+ c'_1 + c_1p$  holds.

\end{corE}

\begin{proof}

$(\Rightarrow)$ Suppose $y$ is a smooth point.  Then $\delta_y=\mathrm{dim}_k ( \tilde{\mathcal{O}_y} / \mathcal{O}_y )=0$ and $r_y=1$ since there is one branch and so $g_y = \delta_y - r_y +1=0-1+1=0$ . If $g_y=0$ in the unibranched case then, by the previous proposition, $p(r_1+r_2-c_1-1)-c'_1-1=0$ which rearranges to $p(r_1+r_2-1)=1+ c'_1 + c_1p$.

$(\Leftarrow) $ Suppose that we are in case 1 and $p(r_1+r_2-1)=1+ c'_1 + c_1p$, then $g_y=0$. As there is one branch $r_y=1$ and 
so we have that $\delta_y=g_y + r_y -1=0+1-1$ which in turn implies $y$ is a smooth point.
\end{proof}

\section*{\S 3. On the existence of a torsor structure}
In this section we discuss the question of the existence of a torsor structure for a Galois cover of type $(p,p)$ between formal normal $R$-schemes.
In addition to the notations set at the beginning of this paper, in this section we allow $R$ to be a {\it complete discrete valuation ring of equal characteristic $p>0$} with 
{\it algebraically closed residue field} $k$.
Let $X$ be a (\emph{formal}) $R$-scheme of finite type which is \emph{normal}, geometrically connected, and flat over $R$. We further assume 
that the special fibre $X_k$ of $X$ is \emph{integral}.
Let $f_K:Y_K\rightarrow X_K$ be an \emph{\'etale torsor} under a finite \'etale $K$-group scheme $\widetilde G$ of rank $p^t$ ($t\ge 1$), 
with $Y_K$ \emph{geometrically connected}, and $f:Y\rightarrow X$ the corresponding 
morphism of \emph{normalisation}. (Thus, $Y$ is the normalisation of $X$ in $Y_K$.)
We are interested in the following question.

\begin{queA} When is $f:Y\rightarrow X$ a \emph{torsor} under a finite and flat $R$-group scheme $G$ which extends $\widetilde G$, i.e., with $G_K=\widetilde G$? 
\end{queA}

The following is well known.

\begin{thmD} \emph{(Proposition 2.4 in [Sa\"\i di2]; Theorem 5.1 in [Tossici])} If $\chr(K)=0$ we assume that $X$ is \emph{locally factorial}. Let $\eta$ be the generic point of $X_k$ and ${\mathcal{O}}_{\eta}$ the local ring of $X$ at $\eta$, which is a discrete valuation ring with fraction field $K(X)$: the function field of $X$. Let $f_K : Y_K \rightarrow X_K$ be an \'etale torsor under a finite \'etale $K$-group scheme $\widetilde G$ of {\bf rank} $\bold p$, with $Y_K$ connected, and let $K(X) \rightarrow L$ be the corresponding extension of function fields. Assume that the ramification index above ${\mathcal{O}}_\eta$ in the field extension  $K(X) \rightarrow L$ equals 1. Then $f : Y \rightarrow X$ is a torsor under a finite and flat $R$-group scheme $G$ of rank $p$ which extends $\widetilde{G}$ (i.e., with $G_K=\widetilde{G}$). 
\end{thmD}

Strictly speaking the above references treat the case where $\chr(K)=0$. For the equal characteristic $p>0$ case see
[Sa\"\i di3], Theorem 2.2.1. Theorem 3.2 also holds when $X$ is the formal spectrum of a complete discrete valuation ring (cf. 
[Sa\"\i di1], Proposition 2.3, and the references therein in the unequal characteristic case, as well as Proposition 2.3.1 in 
[Sa\"\i di4] in the equal characteristic $p>0$ case). 
It is well known that the analog of Theorem 3.2 is \emph{false} in general. 
There are counterexamples to the statement in Theorem 3.2 where $\widetilde G$ is cyclic of rank $p^2$, 
see [Tossici], Example 6.2.12, for instance.

Next, we describe the setting in this section. Let $n\ge 1$, and for  $i\in \{1,\cdots,n\}$ let 
$$f_{i,K}:X_{i,K} \rightarrow X_K$$ 
be an \'etale torsor under an \'etale finite \emph{commutative} $K$-group scheme $\widetilde G_i$, with $X_{i,K}$ \emph{geometrically connected},
such that the $\{f_{i,K}\}_{i=1}^n$ are \emph{generically pairwise disjoint}, i.e. $f_{i,K}$ and $f_{j,K}$ are generically disjoint for $i\neq j$. 
Assume that $f_{i,K}:X_{i,K} \rightarrow X_K$ extends to a torsor
$$f_i:X_i\rightarrow X$$ 
under a finite and flat (necessarily commutative) $R$-group scheme $G_i$ with $(G_i)_K=\widetilde G_i$, 
and with $X_i$ \emph{normal}, $\forall i\in \{1,\cdots,n\}$. (Thus, $X_i$ is the normalisation of $X$ in $X_{i,K}$.)
Let 
$$\widetilde {X}_K\defeq X_{1,K}\times _{X_K}X_{2,K}\times _{X_K} \cdots \times _{X_K}X_{n,K},$$ 
and $\widetilde X$ the \emph{normalisation} of $X$ in $\widetilde X_K$. Thus, $\widetilde X_K$ is the generic fibre
of $\widetilde X$ and we have the following commutative diagrams

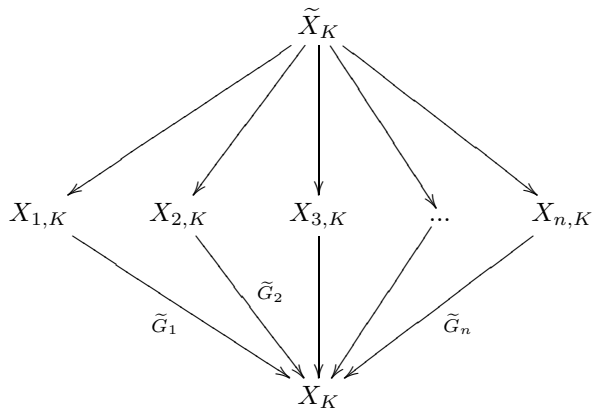
\begin{equation*}
 \xymatrix{
 && \widetilde{X}_K \ar@{->}[ddll] \ar@{->}[ddl] \ar@{->}[dd] \ar@{->}[ddr] \ar@{->}[ddrr] \\
  \\
X_{1,K} \ar@{->}[ddrr]_{\widetilde{G}_{1}} & X_{2,K} \ar@{->}[ddr]^{\widetilde{G}_{2}}  & X_{3,K} \ar@{->}[dd] & ... \ar@{->}[ddl] & X_{n,K} \ar@{->}[ddll]^{\widetilde{G}_{n}} \\
  \\
&& X_K 
}
\end{equation*}

and
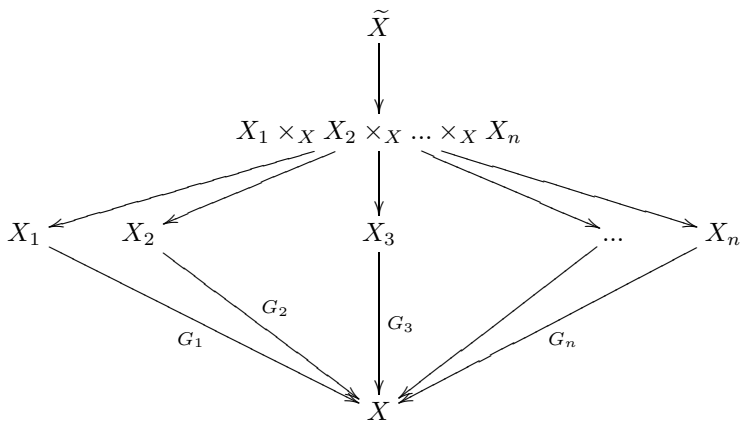
\begin{equation*}
 \xymatrix{
&& \widetilde{X} \ar@{->}[d]  \\
&& X_{1} \times_X X_{2} \times_X  ... \times_X X_{n} \ar@{->}[dll] \ar@{->}[dl] \ar@{->}[d] \ar@{->}[dr] \ar@{->}[drr]  \\
X_{1} \ar@{->}[ddrr]_{G_{1}} & X_{2} \ar@{->}[ddr]^{G_{2}}  & X_{3} \ar@{->}[dd]^{G_{3}} & ...  \ar@{->}[ddl] & X_{n} \ar@{->}[ddll]^{G_{n}} \\
\\
&& X
}
\end{equation*}
where $X_{1} \times_X X_{2} \times_X \cdots \times_X X_{n}$ denotes the fibre product of the $\{X_i\}_{i=1}^n$ over $X$, the morphism 
$\widetilde X\rightarrow X_{1} \times_X X_{2} \times_X \cdots \times_X X_{n}$ is birational and is induced by the natural \emph{ finite} morphisms
$\widetilde X\rightarrow X_{i}$, $\forall i\in \{1,\cdots,n\}$. Note that 
$f_K:\widetilde X_K\rightarrow X_K$ (resp. $\tilde f:X_{1} \times_X X_{2} \times_X \cdots \times_X X_{n} \rightarrow X$) is a torsor 
under the \'etale finite commutative $K$-group scheme $\widetilde G\defeq \widetilde G_1 \times_{\Spec K} \widetilde G_2 \times_{\Spec K} \cdots 
\times_{\Spec K} \widetilde G_n$
(resp. a torsor under the finite and flat commutative $R$-group scheme $G_1 \times_{\Spec R} G_2 \times_{\Spec R} \cdots \times_{\Spec R} G_n$), 
as follows easily from the various definitions. Note that $\left (G_1 \times_{\Spec R} G_2 \times_{\Spec R} \cdots \times_{\Spec R} G_n 
\right )_K=\widetilde G$. 

In this setup Question 1 reads as follows.

\begin{queB} When is $f:\widetilde{X} \rightarrow X$ a torsor under a finite and flat (necessarily commutative) $R$-group scheme 
$G$ which extends $\widetilde G$, i.e., with $G_K=\widetilde G$?
\end{queB}

Our main result in this paper is the following. 

\begin{thmH} We use the same notations as above. Assume that $\widetilde{X}_k$ is {\bf reduced}. Then the following three statements are equivalent.
\begin{enumerate}
\item $f:\widetilde{X}\rightarrow X$ is a torsor under a finite and flat commutative $R$-group scheme $G$, in which case  $G=G_1 \times_{\Spec R} \cdots \times_{\Spec R} G_n$ necessarily.
\item $\widetilde{X}= X_{1} \times_X X_{2} \times_X  \cdots \times_X X_{n}$, in other words 
$X_{1} \times_X X_{2} \times_X  \cdots \times_X X_{n}$ is normal.
\item $\left( X_{1} \times_X X_{2} \times_X \cdots\times_X X_{n} \right)_k$ is reduced.
\end{enumerate}
\end{thmH}

Note that the above condition in Theorem 3.4 that $\widetilde{X}_k$ is reduced is always satisfied after possibly passing to a finite extension $R'/R$ of $R$ (cf. [Epp]).  It implies that the $(X_i)_k$ are reduced, $\forall i\in \{1,\cdots,n\}$. Moreover, Theorem 3.2 and Theorem 3.4 provide a ``complete" answer to Question 1 in the case of Galois covers of type $\left (p,\cdots,p\right )$, i.e., the case where $\mathrm{rank} (G_i)=p,\ \forall i\in \{1,\cdots,n\}$.
 
In the case of (relative) \emph{smooth curves} one can prove the following more precise result when $\mathrm{rank} (G_i)=p,\ \forall i\in \{1,\cdots,n\}$.
.

\begin{thmG}  We use the same notations and assumptions as in Theorem 3.4. Assume further that $X$ is a (relative) \textbf{smooth $R$-curve}, 
$n\ge 2$, and $\mathrm {rank} (G_i)=p$ for $1\le i\le n$. Then the three (equivalent) conditions in Theorem 3.4 are equivalent to the following.

4. \textbf{At least} $\textbf{n-1}$ of the finite flat $R$-group schemes $G_{i}$ acting on $f_i:X_i \rightarrow X$ are \textbf{ \'{e}tale}, for $i \in \{1,\cdots,n\}$.
\end{thmG}

\medskip
\noindent
{\bf Remark 3.6.}\  {\bf 1)}\ Theorem 3.4 holds true if $X$ is the formal spectrum of a complete discrete valuation ring 
(cf. the details of the proof of Theorem 3.4 below which applies as it is in this case).

\noindent
{\bf 2)}\ In $\S3$ we provide examples showing that Theorem 3.5 \emph{doesn't hold} in relative dimension $>1$.

\section*{Proof of Theorem 3.4}

Next, we prove Theorem 3.4. We start by the following.

\begin{proB}
\emph{Let $G$ be a finite and flat commutative $R$-group scheme whose generic fibre is a product of group schemes of 
the form
$ G_K=\widetilde G_1 \times_{\Spec K}\widetilde G_2 \cdots  \times_{\Spec K} \widetilde G_n,$
where the $\{\widetilde G_i\}_{i=1}^n$ are finite and flat commutative $K$-group schemes. Then $G$ is a product of 
finite and flat commutative $R$-group schemes  $\{G_i\}_{i=1}^n$, i.e.,
$ G = G_1 \times_{\Spec R}  G_2 \times_{\Spec R}\cdots \times_{\Spec R} G_n,$
with $(G_{i})_K= \widetilde G_i$.}
\end{proB}

\begin{proof} First, we treat the case $n=2$. Thus, we have $G_K=\widetilde G_1 \times_{\Spec K} \widetilde G_2$ and need 
to show $G = G_1 \times_{\Spec R} G_2$ where $(G_{i})_K= \widetilde G_i$, for $i=1,2$. 
Let $G_i$ be the \emph{ schematic closure} of $\widetilde G_{i}$ in $G$, for $i=1,2$ (cf. [Raynaud], 2.1). Therefore, $G_1$ and $G_2$ are 
closed subgroup schemes of $G$ which are finite and flat over $\Spec R$ (cf. loc. cit.). 
We have a short exact sequence
$$ 1  \rightarrow  G_1  \rightarrow  G  \rightarrow  G/G_1  \rightarrow 1 ,$$
and likewise
$$ 1  \rightarrow  G_2  \rightarrow  G  \rightarrow  G/G_2  \rightarrow 1 ,$$
of finite and flat commutative $R$-group schemes (cf. loc. cit.).
It remains for the proof to show that the composite homomorphism $G_2 \rightarrow G \rightarrow G/G_1$ is 
an isomorphism. The morphism $G \rightarrow G/G_1$ is finite. The morphism $G_2 \rightarrow G$ is a closed 
immersion, hence finite. The composite $G_2 \rightarrow G/G_1$ of the above morphisms is then finite. 
We will show it is an isomorphism. The morphism $G_2 \rightarrow G/G_1$ 
is a closed immersion since its kernel is trivial. Indeed, on the generic fibre the kernel is trivial: 
$\left(G_1 \cap G_2\right)_K = \widetilde{G_1} \cap \widetilde{G_2}  = \{1\}$. The map $G_2 \rightarrow G/G_1$ is then an 
isomorphism as both group schemes have the same rank. Similarly, the morphism $G_1 \rightarrow G/G_2$ is an isomorphism. 
Therefore, $G=G_1 \times_{\Spec R} G_2$ as required. Now an easy devissage argument along the above lines of thought, using 
induction on $n$, reduces immediately to the above case $n=2$.
\end{proof}

\emph{Proof of Theorem 3.4}

\begin{proof}
(1 $\Rightarrow$ 2)
Assume that $f:\widetilde{X}\rightarrow X$ is a torsor under a finite and flat $R$-group scheme $G$.
In particular, $G_K=\widetilde G$ and $G$ is necessarily commutative.
We will show that $\widetilde{X}= X_{1} \times_X X_{2} \times_X  ... \times_X X_{n}$, i.e., show that 
$X_{1} \times_X X_{2} \times_X  ... \times_X X_{n}$ is normal (this will imply that  $G=G_1 \times_{\Spec R}\cdots \times_{\Spec R} G_n$
necessarily, as $G_1 \times_{\Spec R} ... \times_{\Spec R} G_n$ is the group scheme of the torsor $\tilde f:X_{1} \times_X X_{2} \times_X\cdots 
\times_X X_{n}\rightarrow X$). One reduces easily by a devissage argument to the case $n=2$ 
which we will treat below.

Assume $n=2$. We have the following commutative diagrams of torsors

\begin{equation*}
 \xymatrix{
& \widetilde{X}_K  \ar@{->}[dl]_{\widetilde{G}_{2}} \ar@{->}[dr]^{\widetilde{G}_{1}}  \\
X_{1,K} \ar@{->}[dr]_{\widetilde{G}_{1}}  &&  X_{2,K}  \ar@{->}[dl]^{\widetilde{G}_{2}}   \\
& X_K 
}
\end{equation*}

and

\begin{equation*}
 \xymatrix{
& \widetilde{X} \ar@{->}[d]   \ar@/^2pc/[ddr]^{G'_1}  \ar@/_2pc/[ddl]_{G'_2}   \\
& X_{1} \times_X X_{2}  \ar@{->}[dl]_{G_2  \ \ } \ar@{->}[dr]^{ \ \ G_1}  \\
X_{1} \ar@{->}[ddr]_{G_1}  &&  X_2  \ar@{->}[ddl]^{G_2}   \\
\\
& X 
}
\end{equation*}
where $\widetilde X\rightarrow X_i$ is a torsor under a finite and flat $R$-group scheme $G_j'$, for $j\neq i$.
Moreover, $G'_1 = \left( \widetilde G_{1} \right)^{\text{schematic closure}}$, and $G'_2 = \left( \widetilde G_{2} 
\right)^{\text{schematic closure}}$ (where the schematic closure is taken inside $G$) holds necessarily, so that $G=G'_1 \times_{\Spec R} G'_2$ (cf. Proposition 1.1).
Note that $\widetilde X/{G'_1}=X_2$ must hold as the quotient $\widetilde X/{G'_1}$ is normal: since  
$\left (\widetilde X/{G'_1}\right )_k$ is reduced (as $\widetilde X_k$ is reduced and $\widetilde X$ dominates $\widetilde X/{G'_1}$),
and $\left (\widetilde X/{G'_1}\right )_K=X_{2,K}$
is normal (cf. [Liu], 4.1.18). Similarly $\widetilde X/{G'_2}=X_1$ holds.
We want to show that $\widetilde{X}= X_{1} \times_X X_{2}$, and we claim that this reduces to showing that the natural morphism $G \to G_1 \times_{\Spec R} G_2$ 
(cf. the map $\phi$ below) is an isomorphism. Indeed, if one has two torsors, in this case $ \widetilde{X}\rightarrow X$
and $X_{1} \times_X X_{2}\rightarrow X$ above the same $X$, under isomorphic group schemes, which are isomorphic on the generic fibres, and if we have a morphism 
$\widetilde{X} \rightarrow X_{1} \times_X X_{2}$ which is compatible with the torsor structure and the given identification of group schemes (cf. above diagrams and the definition of $\phi$ below), 
then this morphism must be an isomorphism. (This is a consequence of Lemma 4.1.2 in
[Tossici]. In [Tossici] $\chr(K)=0$ is assumed, the same proof however applies if $\chr(K)=p$.) 
We have two short exact sequences of finite and flat commutative $R$-group schemes
(cf. above diagrams and discussion for the equalities  $G_1=G/G'_2$ and  $G_2=G/G'_1$)
$$ 1  \rightarrow G'_2  \rightarrow G  \rightarrow G_1 = G/G'_2  \rightarrow 1 ,$$
and 
$$ 1  \rightarrow G'_1  \rightarrow G  \rightarrow G_2 = G/G'_1  \rightarrow 1 .$$
The morphisms $G \rightarrow G_1$, and $G \rightarrow G_2$, are finite. Consider the following exact sequence

$$ 1  \rightarrow \mathrm{Ker}(\phi) \rightarrow G \rightarrow G_1 \times_{\Spec R} G_2 ,$$
where $\phi:G   \rightarrow G_1 \times_{\Spec R} G_2$ is the morphism induced by the above morphisms.
We want to show that the map $\phi:G \rightarrow G_1 \times_{\Spec R} G_2$ is an isomorphism. We have $\mathrm{Ker}(\phi)=G'_1 \cap G'_2$ 
by construction. However, $G'_1 \cap G'_2 = \{ 1 \}$ since $G=G'_1 \times _{\Spec R}G'_2$ by Proposition 1.1, and therefore 
$\mathrm{Ker}(\phi) = \{ 1 \}$ which means $\phi:G \rightarrow G_1 \times_{\Spec R} G_2$  is a closed immersion. Finally, $G$ and 
$G_1 \times_{\Spec R} G_2$ have the same rank as group schemes which implies $\phi$ is an isomorphism, as required. 

(2 $\Rightarrow$ 3) Clear.

(3 $\Rightarrow$ 1) By assumption $\left(X_{1} \times_X X_{2} \times_X  ... \times_X X_{n} \right)_k$ is reduced. 
Moreover, we have 

\noindent
$\left( X_{1} \times_X X_{2} \times_X  ... \times_X X_{n} \right)_K=\widetilde X_K$ is normal. Hence 
$X_{1} \times_X X_{2} \times_X  ... \times_X X_{n}$ is normal (cf. [Liu], 4.1.18),
and $\widetilde{X} = X_{1} \times_X X_{2} \times_X  ... \times_X X_{n}$. We know 
that $\tilde f:X_{1} \times_X X_{2} \times_X  ... \times_X X_{n}\rightarrow X$ 
is a torsor under the group scheme 
$G_1 \times_{\Spec R} G_2 \times_{\Spec R} .... \times_{\Spec R} G_n$, 
so $f:\widetilde{X}\rightarrow X$ is a torsor under the same group scheme.
\end{proof}

\section*{Proof of Theorem 3.5}
Next, we prove Theorem 3.5.

\begin{proof}

(1 $\Rightarrow$ 4) Suppose that $\tilde f:\widetilde{X}\rightarrow X$ is a torsor under a finite and flat $R$-group scheme $G$; 
in which case $\widetilde{X} = X_{1} \times_X X_{2}\times_X  ... \times_X X_{n}$ and $G=G_1 \times_{\Spec R} \cdots \times_{\Spec R} G_n$ 
(cf. Theorem 3.4). We will show that \emph{at least} ${n-1}$ of the finite 
flat $R$-group schemes $G_{i}$ (acting on $f_i:X_i \rightarrow X$) are \'{e}tale, for $i\in \{1,\cdots,n\}$.
We argue by induction on the rank of $G$.

\emph{ Base case:} The base case pertains to  $\mathrm{rank} (G)=p^2$ and $n=2$. Thus, $\mathrm{rank} (G_1)=\mathrm{rank} (G_2)=p$.
We assume $\widetilde{X} = X_{1} \times_X X_{2}$ and prove that at 
least one of the two group schemes $G_1$ or $G_2$ is \'{e}tale. We assume that $X$ is a scheme, and not a formal scheme, in which case the argument of proof is the same.

Let $x\in X$ be a \emph{closed} point of $X$ and $\mathcal{X}$ the \emph{boundary of the formal germ} of $X$ at $x$, so $\mathcal{X}$  is isomorphic to 
$\mathrm{Spec} \left( R[[T]] \{ T^{-1} \} \right)$ (cf. Background). We have a natural morphism $\mathcal{X}\rightarrow X$ of schemes.
Write $\mathcal{X}_1\defeq \mathcal{X}\times _XX_1$, $\mathcal{X}_2\defeq \mathcal{X}\times _XX_2$, and $\widetilde {\mathcal{X}}\defeq 
\mathcal{X}\times _X\widetilde X$. Thus, by base change, $\widetilde {\mathcal{X}}\rightarrow \mathcal{X}$ (resp. $\mathcal{X}_1\rightarrow \mathcal{X}$, and 
$\mathcal{X}_2\rightarrow \mathcal{X}$) is a torsor under the group scheme $G$ (resp. under $G_1$, and $G_2$) and we have the following commutative diagram 

\begin{equation*}
 \xymatrix{
& \widetilde{\mathcal{X}} = \mathcal{X}_{1} \times_\mathcal{X} \mathcal{X}_{2}  \ar@{->}[dl]_{G_{2}} \ar@{->}[dr]^{G_{1}}  \\
\mathcal{X}_{1} \ar@{->}[dr]_{G_{1}}  &&  \mathcal{X}_2  \ar@{->}[dl]^{G_{2}}   \\
& \mathcal{X}
}
\end{equation*}
Note that 
$\widetilde {\mathcal{X}}$ is normal as $(\widetilde {\mathcal{X}})_k$ is reduced (recall $(\widetilde X)_k$ 
is reduced) and $(\widetilde {\mathcal{X}})_K$ 
is normal (cf. [Liu], 4.1.18), hence $\widetilde {\mathcal{X}}=\mathcal{X}_1\times _{\mathcal{X}}\mathcal{X}_2$ holds 
(cf. Theorem 3.4 and Remarks 3.6, 1).

Assume now that $G_{1}$ and $G_{2}$ are both \emph{non-\'{e}tale} $R$-group schemes. Then we prove that  
$\widetilde {\mathcal{X}}\rightarrow \mathcal{X}$ can not have the structure of a torsor under a finite and flat $R$-group scheme which 
would then be a contradiction. More precisely, we will prove that $\mathcal{X}_1\times _{\mathcal{X}}\mathcal{X}_2$ can not be normal in this case, 
hence the above conclusion (cf. Theorem 3.4).

We will assume for simplicity that $\chr(K)=0$ and $K$ contains a primitive $p$-th root of $1$.
A similar argument as the one used below holds in equal characteristic $p>0$.
First, $\widetilde {\mathcal{X}}$ is connected as $\widetilde X_k$ is \emph{unibranch} (the finite morphism $\widetilde X_k\to X_k$ is radicial).
As the group schemes $G_1$ and $G_2$ are non  \'{e}tale, their special fibres $(G_{1})_k$ and $(G_{2})_k$ are radicial isomorphic to either 
$\mu_p$ or $\alpha_p$. We treat the case  $(G_{1})_k$ is isomorphic to $\mu_p\defeq \mu_{p,k}$ and $(G_{2})_k$ is isomorphic to $\alpha_p\defeq \alpha_{p,k}$; 
the remaining cases are treated similarly. 
Recall $\mathcal{X}$ is isomorphic to $\mathrm{Spec} \left(R[[T]] \{ T^{-1} \} \right)$. For a suitable choice of the parameter $T$ 
the torsor $\mathcal{X}_2\rightarrow \mathcal{X}$ is given by an equation $Z_2^p=1+\pi^{np}T^m$
where $n$ is a positive integer (satisfying a certain condition) and $m\in \mathbb{Z}$ (cf. Background. Also see Proposition 2.3.1 in [Sa\"\i di4] for the equal characteristic case),
and the torsor $\mathcal{X}_1\rightarrow \mathcal{X}$ is given by an equation $Z_1^p=f(T)$ where $f(T)\in R[[T]] \{ T^{-1} \}$ is a unit whose 
reduction $\overline {f(T)}$ modulo $\pi$ is not a $p$-power (cf. loc. cit.). We claim that 
$\widetilde{\mathcal{X}} = \mathcal{X}_{1} \times_{\mathcal{X}} {\mathcal{X}}_{2}$ can not hold. Indeed, by base change ${\mathcal{X}}_{1} \times_{\mathcal{X}} \mathcal{X}_{2}\rightarrow 
\mathcal{X}_2$ is a $G_1$-torsor 
which is generically given by an equation  $Z^p=f(T)$, where $f(T)$ is viewed as a function on $\mathcal{X}_2$. 
But in $\mathcal{X}_2$ the function $T$ becomes a $p$-power modulo $\pi$
as one easily deduces from the equation $Z_2^p=1+\pi^{np}T^m$ defining the torsor $\mathcal{X}_2\rightarrow \mathcal{X}$. 
Indeed, after a change of variables we can write the above equation as $(1+\pi^nZ_2')^p=1+\pi^{np}T^m$ which reduces, after an easy computation, to an equation $z_2'^p=t^m$
hence $({(z_2')}^{\frac {1}{m}})^p=t$. In particular, the   
reduction $\overline {f(T)}$ modulo $\pi$ of $f(T)$, viewed as a function on $\left (\mathcal{X}_2\right )_k$, is a $p$-power. This means that 
$({\mathcal{X}}_{1} \times_{\mathcal{X}} \mathcal{X}_{2})_k$ is not reduced and
$\widetilde {\mathcal{X}}\rightarrow \mathcal{X}_2$ can not be a $G_1\simeq \mu_{p,R}$-torsor (cf. the proof of Proposition 2.3 in [Sa\"\i di1]), and a fortiori
$\widetilde{\mathcal{X}} \neq {\mathcal{X}}_{1} \times_{\mathcal{X}} \mathcal{X}_{2}$.

\emph{ Inductive hypothesis: } Given $G$, we assume that the (1 $\Rightarrow$ 4) part in Theorem 3.5 holds true for smaller values of $n \ge 2$. 
Then $\widetilde{X}_1 \defeq X_{1} \times_X X_{2} \times_X  ... \times_X X_{n-1}$ 
is normal (since its special fibre is reduced (as it is dominated by $\widetilde{X}$ whose special fibre is reduced)
and its generic fibre is normal (cf. [Liu], 4.1.18)), hence at least $n-2$ of the 
corresponding $G_{i}$'s, for $i\in \{1,\cdots,n-1\}$, are \'{e}tale by the induction hypothesis. We will assume, without loss of generality, 
that $G_{i}$ is \'etale for $1 \leq i \leq n-2$.

\emph{ Inductive step: } We have the following picture for our inductive step (the case for $n$):

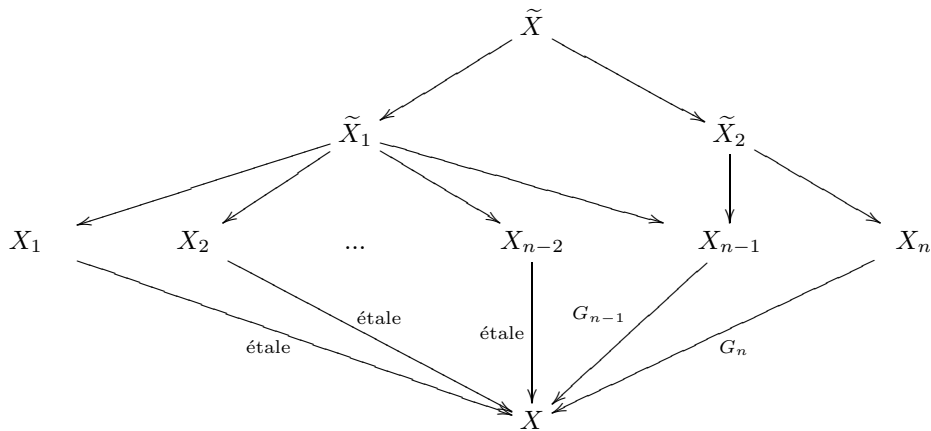
\begin{equation*}
 \xymatrix{
&&& \widetilde{X}\\
&&  \widetilde{X}_1  \ar@{<-}[ur]  && \widetilde{X}_2  \ar@{<-}[ul]  \\
\ \ \ X_{1} \ \ \ \ar@{<-}[urr]  & \ \ \  X_{2} \ \ \  \ar@{<-}[ur]  & \ \ \ ... \ \ \ & \ \ \ X_{n-2} \ \ \ \ar@{<-}[ul]  & \ \ \ X_{n-1} \ \ \ \ar@{<-}[ull] \ar@{<-}[u] & \ \ \ X_n \ \ \  \ar@{<-}[ul] \\
\\
&&& X   \ar@{<-}[uulll]^{\text{\'etale}}  \ar@{<-}[uull]_{\text{\'etale}} \ar@{<-}[uu]^{\text{\'etale}}  \ar@{<-}[uur]^{G_{n-1}}   \ar@{<-}[uurr]_{G_{n}}
}
\end{equation*}

We argue by contradiction. Suppose that neither $G_{n-1}$ nor $G_{n}$ is \'{e}tale.  This would mean that $\widetilde{X}_2\rightarrow X$, where $\widetilde{X}_2$ is
the normalisation of $X$ in $(X_{n-1})_K \times_{X_K} (X_{n})_K$, does not have the structure of a torsor (as this would contradict the induction hypothesis).
This implies that $\widetilde{X}\rightarrow X$ does not have the structure of a torsor since it factorises 
$\widetilde{X}\rightarrow \widetilde{X}_2\rightarrow X$, for otherwise $\widetilde{X}_2\rightarrow X$ being a quotient of 
$\widetilde{X}\rightarrow X$ would be a torsor. 
Of course, $\widetilde{X}\rightarrow X$ is a torsor to start with by assumption and so this is a contradiction. 
Therefore, at least one of $G_{n-1}$ and $G_{n}$ is \'{e}tale, as required.

(1 $\Leftarrow$ 4)  Suppose that at least $n-1$ of the $G_{i}$ are \'{e}tale, say: $G_1,G_2,\cdots,G_{n-1}$ are \'etale. Write 
$\widetilde X_1\defeq  X_{1} \times_X X_{2}\times_X  ... \times_X X_{n-1}$. Then $\widetilde X_1\rightarrow X$ is a torsor under the 
finite {\it \'etale} $R$-group scheme 
$G_1'\defeq G_1\times _{\Spec R}G_2\times_{\Spec R}\cdots\times _{\Spec R}G_{n-1}$. Moreover, $X_{1} \times_X X_{2}\times_X  ... \times_X X_{n}
=\widetilde X_1\times_XX_n$, 
and $X_{1} \times_X X_{2}\times_X  ... \times_X X_{n}\rightarrow X_{n}$ is an \'etale torsor under the group scheme $G'_1$ (by base change). 
In particular, $\left (X_{1} \times_X X_{2}\times_X  ... 
\times_X X_{n}\right )_k$ is reduced as $(X_{n})_k$ is reduced.
Indeed, $\widetilde X$ dominates $X_n$ and ${\widetilde X}_k$ is reduced. 
Hence $\widetilde X=X_{1} \times_X X_{2}\times_X  ... \times_X X_{n}$ (cf. Theorem 3.4) and 
$\widetilde X\rightarrow X$ is a torsor under the group scheme $G\defeq G_1\times _{\Spec R}G_2\times_{\Spec R}\cdots\times _{\Spec R}G_{n}$.
\end{proof}

\section*{3.8. Counterexample to Theorem 3.5 in higher dimensions}

Theorem 3.5 is not valid (under similar assumptions) for (formal) smooth $R$-schemes of relative dimension $\geq 2$. Here is a counterexample.
Assume $\chr(K)=0$ and $K$ contains a primitive $p$-th root of $1$. Let $X= \mathrm{Spf}(A)$ where $A\defeq R<T_1, T_2>$ is the free 
$R$-Tate algebra in the two variables $T_1$ and $T_2$. 
Let $G_1=G_2=\mu_p\defeq \mu_{p,R}$, neither being an \'etale $R$-group scheme. For $i=1,2$, consider the 
$G_i$-torsor $X_i\rightarrow X$ which is generically defined by the equation 
$$Z_i^p=T_i.$$ 
We have the following commutative diagram

\begin{equation*}
 \xymatrix{
&   X_1 \times X_2  \ar@{->}[ddl]^{(Z'_2)^p=T_2}_{\mu_p} \ar@{->}[ddr]_{(Z'_1)^p=T_1}^{\mu_p}  \\
\\
X_1   \ar@{->}[ddr]_{\mu_p}^{Z_1^p=T_1}  && X_2  \ar@{->}[ddl]^{\mu_p}_{Z_2^p=T_2} \\
\\
&  X= \mathrm{Spf} \left( R<T_1, T_2> \right) \\
}
\end{equation*}

The torsor $X_1\times_XX_2\rightarrow X_2$ is a $G_1=\mu_p$-torsor defined generically by the equation 
$$(Z'_1)^p=T_1$$ 
where $T_1$ is viewed as a function on $X_2$.
This function is not a $p$-power modulo $\pi$ as follows easily from the fact that the torsor $X_2\rightarrow X$ is defined 
generically by the equation $Z_2^p=T_2$. In particular, $X_1\times_XX_2\rightarrow X_2$ is a non trivial $\mu_p$-torsor, and
$(X_1\times_XX_2)_k\rightarrow (X_2)_k$ is a non trivial $\mu_{p,k}$-torsor. Hence
$(X_1\times_XX_2)_k$ is necessarily reduced (as $(X_2)_k$ is reduced since $(X_2)_k\to (X_1)_k$ is a non trivial $\mu_{p,k}$-torsor). 
Thus,  $X_1\times_XX_2$ is normal (cf. Theorem 3.4) and 
$X_1\times_XX_2=\widetilde X$, where $\widetilde X$ is the normalisation of $X$ in 
$(X_1\times_XX_2)_K$, which contradicts the statement of Theorem 3.5 in this case.

$$\textbf{References.}$$

\noindent
[Bourbaki] Bourbaki, N. Alg\`ebre Commutative, Chapitre 9, Masson, 1983.

\noindent
[Epp] Epp, H.P., Eliminating wild ramification, Inventiones Mathematicae, Volume 19 (1973), pp. 235-249.


\noindent
[Kato] Kato, K. Vanishing cycles, ramification of valuations, and class field theory, Duke Mathematical Journal 55 (3) (1987), pp. 629-659.

\noindent
[Liu] Liu, Q., Algebraic Geometry and Arithmetic Curves, Oxford University Press, 2009.

\noindent
[Raynaud] Raynaud, M., Sch\'emas en groupes de type $(p,...,p)$, Bulletin de la Soci\'et\'e Math\'ematique de France, 
Volume 102, 1974, pp. 241-280.

\noindent
[Sa\"\i di1] Sa\"\i di, M., Wild ramification and a vanishing cycles formula, Journal of Algebra 273 (2004) 108-128.

\noindent
[Sa\"\i di2] Sa\"\i di, M., Torsors under finite and flat group schemes of rank $p$ with Galois action, Mathematische 
Zeitschrift, Volume 245, 2003, pp. 695-710.

\noindent
[Sa\"\i di3] Sa\"\i di, M., On the degeneration of \'etale $\Bbb Z/p\Bbb Z$ and $\Bbb Z/p^2\Bbb Z$-torsors in equal characteristic $p>0$, Hiroshima Math. J. 37 (2007), 315-341.

\noindent
[Sa\"\i di4] Sa\"\i di, M., Galois covers of degree $p$ and semi-stable reduction of curves in equal characteristic $p>0$,
Math. J. Okayama Univ. 49 (2007), 113-138.

\noindent
[Serre] Serre, J.-P., Local fields, Springer Verlag, 1979.

\noindent
[Tossici] Tossici, D., Effective models and extension of torsors over a discrete valuation ring of unequal characteristic, International Mathematics Research Notices, 2008, pp. 1-68.

\bigskip

\noindent
Mohamed Sa\"\i di

\noindent
College of Engineering, Mathematics, and Physical Sciences

\noindent
University of Exeter

\noindent
Harrison Building

\noindent
North Park Road

\noindent
EXETER EX4 4QF

\noindent
United Kingdom

\noindent
M.Saidi@exeter.ac.uk

\bigskip

\noindent
Nicholas Williams

\noindent
London, United Kingdom

\noindent
Nicholas.Williams09@alumni.imperial.ac.uk

\end{document}